\documentclass[11pt]{amsart}

\usepackage{amsmath, amsthm, amssymb}
  
\usepackage{lmodern}
\usepackage{palatino}                       
\usepackage[bigdelims,vvarbb]{newpxmath}    
\usepackage[scaled=0.95]{inconsolata}       
\usepackage[T1]{fontenc}                    

\usepackage{comment}

\usepackage[abs]{overpic}		  
\usepackage[all,cmtip]{xy}

\usepackage{xcolor}
\definecolor{indigo}{rgb}{0.29, 0.0, 0.51}  
\usepackage[colorlinks, urlcolor=indigo, linkcolor=indigo, citecolor=indigo]{hyperref}

\usepackage[hcentering, vcentering, total={5.9in, 8.2in}]{geometry}  

\usepackage{tikz}
\usetikzlibrary{trees}

\usepackage{mathtools}

\theoremstyle{plain}
\newtheorem{theorem}{Theorem}
\newtheorem{corollary}[theorem]{Corollary}
\newtheorem{proposition}[theorem]{Proposition}
\newtheorem{lemma}[theorem]{Lemma}
\newtheorem{question}[theorem]{Question}

\theoremstyle{definition}
\newtheorem{definition}[theorem]{Definition}

\theoremstyle{remark}
\newtheorem{remark}[theorem]{Remark}
\newtheorem{example}[theorem]{Example}

\newtheorem{construction}[theorem]{Construction}

\numberwithin{theorem}{section}

\newcommand{\dfn}[1]{{\em #1}}        
\DeclareMathOperator{\bd}{\partial}   

\newcommand{\R}{\mathbb{R}}           
\newcommand{\Z}{\mathbb{Z}}           
\newcommand{\C}{\mathbb{C}}           
\newcommand{\NS}{{\mathbb{S}}}
\newcommand{\T}{{\mathbb{T}}}
\newcommand{\D}{{\mathbb{D}}}

\newcommand{\Op}{{\mathcal{O}p}}

\newcommand{\K}{\mathcal{K}}           
\renewcommand{\L}{\mathcal{L}}           
\newcommand{\FL}{\mathcal{FL}}           

\makeatletter
\newcommand*\bigcdot{\mathpalette\bigcdot@{0.6}}
\newcommand*\bigcdot@[2]{\mathbin{\vcenter{\hbox{\scalebox{#2}{$\m@th#1\bullet$}}}}}
\makeatother

\DeclareMathOperator\twb{\overline {\tw}}                 

\DeclareMathOperator{\PB}{{{PB}}}          

\DeclareFontFamily{U} {cmr}{}
\DeclareFontShape{U}{cmr}{m}{n}{
  <-6> cmr5
  <6-7> cmr6
  <7-8> cmr7
  <8-9> cmr8
  <9-10> cmr9
  <10-12> cmr8
  <12-> cmr9}{}
\DeclareSymbolFont{Xcmr} {U} {cmr}{m}{n}
\DeclareMathSymbol{\Phi}{\mathord}{Xcmr}{8}

\DeclareMathOperator\tbb{\overline {\tb}}     

\mathchardef\mhyphen="2D

\newcommand{\std}{{\operatorname{std}}}

\newcommand{\GL}{\operatorname{GL}}

\newcommand{\Diff}{\operatorname{Diff}}

\newcommand{\SO}{\operatorname{SO}}
\newcommand{\U}{\operatorname{U}}

\newcommand{\Id}{{\operatorname{Id}}}

\newcommand{\rot}{\operatorname{rot}}
\newcommand{\tb}{\operatorname{tb}}

\newcommand{\tw}{\operatorname{tw}}

\newcommand{\Cont}{\operatorname{Cont}}

\newcommand{\Fr}{\operatorname{Fr}}

\newcommand{\Tight}{\operatorname{Tight}}

\newcommand{\image}{\operatorname{Image}}

\newcommand{\kalman}{{K\'{a}lm\'{a}n }}
\newcommand{\kalmans}{{K\'{a}lm\'{a}n's }}

\newcommand{\Emb}{{\mathfrak{Emb}}}

\begin{document} 

\title{Spaces of Legendrian cables and Seifert fibered links}

\author{Eduardo Fern\'{a}ndez}
\address{Department of Mathematics\\ University of Georgia\\ Athens\\ GA}
\email{eduardofernandez@uga.edu}

\author{Hyunki Min}
\address{Department of Mathematics\\ University of Georgia\\ Athens\\ GA}
\email{hkmin27@uga.edu}

    
\begin{abstract}
  We determine the homotopy type of the spaces of several Legendrian knots and links with the maximal Thurston--Bennequin invariant. In particular, we give a recursive formula of the homotopy type of the space of Legendrian embeddings of sufficiently positive cables, and determine the homotopy type of the space of Legendrian embeddings of Seifert fibered links, which include all torus knots and links, in the standard contact $3$-sphere, except when one of the link components is a negative torus knot. In general, we prove that the space of contact structures on the complement of a sufficiently positive Legendrian cable with the maximal Thurston-Bennequin invariant is homotopy equivalent to the space of contact structures on the complement of the underlying Legendrian knot, and prove that the space of contact structures on a Legendrian Seifert fibered space over a compact oriented surface with boundary is contractible. From this result, we find infinitely many new components of the space of Legendrian embeddings in the standard contact $3$-sphere that satisfy an injective $h$-principle. These include the spaces of Legendrian embeddings of an algebraic link with the maximal Thurston--Bennequin invariant. In particular, the inclusion of these Legendrian embedding spaces into the corresponding formal Legendrian embedding spaces is a homotopy injection. 
\end{abstract}

\maketitle

\section{Introduction}\label{sec:intro}

After the foundational work of Bennequin \cite{Bennequin} and Eliashberg \cite{Eliashberg:OT}, there has been a lot of research to understand botany problems in $3$-dimensional contact topology. Such problems usually deal with the classification of tight contact structures up to isotopy, the classification of contactomorphisms up to contact isotopy, and the classification of Legendrian/transverse knots and links up to Legendrian/transverse isotopy. One major question in these botany problems is to understand flexible/rigid phenomena (e.g. overtwisted contact structures follow the $h$-principle, while tight contact structures do not), which were studied a lot  in both directions, see \cite{Chekanov:LCH,Eliashberg:20Years, EliashbergFraser:unknot,EtnyreHonda:torusKnots, Giroux:parametric, GirouxMassot:bundles, Honda:classification1} for examples.

However, little is known about the botany problems for \em families \em of the previous objects. In this article, we focus on the study of families of Legendrian embeddings and contact structures on the complement of knots and links. In this context, regarding flexible/rigid phenomena, we find the following open question particularly interesting:

\begin{question}\label{q:flexibility}
  For $k \geq 1$, is there a non-contractible $k$-sphere of Legendrian embeddings into a contact $3$-manifold that is contractible as a $k$-sphere of formal Legendrian embeddings?
\end{question}

Work of Mart\'inez-Aguinaga, Presas together with the first author \cite{FMP:unknot} implies that the answer is \em negative \em when the underlying Legendrian is an unknot satisfying the relation $\tb+|\rot|=-1$; or a $(n,n)$-torus link with the maximal Thurston--Bennequin number. In this article we will show that the answer of the previous question is \em negative \em for infinitely many new Legendrian knots and links, all of them with the maximal Thurston--Bennequin number.  These Legendrian knots and links are obtained as iterated \em sufficiently positive \em generalized Legendrian cables (see Section~\ref{subsub:generalizedCables} for the definition). They contain every max-$\tb$ Legendrian representative of an isolated plane curve singularity, and every max-$\tb$ Legendrian positive Seifert fibered link, including all torus knots and links. For the latter, we will give a precise description of the homotopy type of the Legendrian embedding spaces. For sufficiently positive Legendrian cable \em knots \em with the maximal Thurston--Bennequin number, we will determine the homotopy type of the corresponding Legendrian embedding space in terms of the homotopy type of the Legendrian embedding space of the underlying Legendrian knot. This can be seen as the contact version of a known recursive formula for the smooth case \cite{Budney:knotsS3}. This, in particular, reduces Question \ref{q:flexibility} about the cable to the question about the underlying knot.

To determine the homotopy type of Legendrian embedding spaces, we will study the space of contact structures on the complements of Legendrian knots and links. In particular, we will show that the space of contact structures on the complement of a max-$\tb$ sufficiently positive generalized Legendrian cable is homotopy equivalent to the space of contact structures on the complement of the underlying Legendrian knots and links. Also, we will show that the space of contact structures on a Legendrian Seifert fibered space (see Definition~\ref{def:lsfs}) over a compact oriented surface with boundary is contractible. From this, we can also determine the homotopy type of the group of contactomorphisms on those Legendrian Seifert fibered spaces, which can be considered as a generalization of the result about the contact mapping class groups of Legendrian circle bundles \cite{GirouxMassot:bundles,FMP:unknot}.   

The study of Legendrian loops has become an important topic in $4$-dimensional symplectic topology in recent years after the foundational work of \kalman \cite{Kalman:exotic} and Ekholm--Honda--\kalman \cite{EHK:Functor}. The reason for this lies in the fact that, as proved in \cite{EHK:Functor}, the monodromy invariant in Legendrian contact homology \cite{Kalman:exotic} of a Legendrian loop, which can be computed from the $3$-dimensional picture in a combinatorial way, gives rise to a Hamiltonian isotopy invariant of the exact Lagrangian concordance associated to the Legendrian loop \cite{Chantraine:lagrangian} in the symplectization of the ambient contact manifold. The classification results provided in this article offer precise lists of Legendrian loops to work with. In \cite{Kalman:exotic}, \kalman used the monodromy invariant to prove the non-triviality of certain loops of Legendrian positive torus knots with the maximal Thurston-Bennequin number. However, all these loops were shown to be non-contractible as loops of formal Legendrians \cite{FMP:Legendrian}, leaving  Question \ref{q:flexibility} open for positive Legendrian torus knots with the maximal Thurston-Bennequin number. In this article, we provide a comprehensive negative answer to the question regarding these Legendrians. In particular, we give an independent proof of the non-contractibility of \kalmans loops as loops of formal Legendrians. 

More recently, Casals and Gao \cite{CasalsHonghao:infnite} considered many other Legendrian loops of some positive torus links with maximal Thurston-Bennequin number to prove the existence of infinitely many Lagrangian fillings in $(\D^4,\omega_\std)$. Again, our work implies that all these loops are also non-contractible at the formal level. In a broader context, Legendrian algebraic links with the maximal Thurston-Bennequin number have been systematically studied by Casals \cite{Casals:PlaneCurve}, who made relevant structural conjectures about the classification of exact Lagrangian fillings, up to Hamiltonian isotopy. Our work provides some control over the map from the loop space of the space of Legendrians to the space of Lagrangian fillings: two loops of Legendrians that are formally homotopic are actually homotopic through Legendrian loops and, therefore, their images are exact Hamiltonian isotopic. 

\subsection{Main results.}

Let $(M,\xi)$ be a tight contact $3$-manifold. Given a smooth knot type $K$ in $M$, we will denote by $\L^{\tbb}(K,(M,\xi))$ the space of Legendrian embeddings realizing the smooth knot type $K$ and with the maximal Thurston--Bennequin invariant. If $(M,\xi)$ is the standard contact $3$-sphere $(\NS^3,\xi_\std)$, we simply write $\L^{\tbb}(K)$ instead of $\L^{\tbb}(K,(\NS^3,\xi_\std))$. We denote the space of smooth embeddings by $\K (K,M)$, the same consideration as before when $M=\NS^3$ will also apply. We will also consider \em long \em embedding spaces, and these are defined as follows. Fix a point $(p_0,v_0)\in\NS(\xi)$ in the unit sphere bundle of $\xi$. The space of long smooth embeddings realizing the knot type $K$ is 
\[
  \K_{(p_0,v_0)}(K,M)=\{\gamma\in\K(K,M):(\gamma(0),\gamma'(0))=(p_0,v_0)\}.
\] 
The space of long Legendrian embeddings realizing the knot type $K$ and with the maximal Thurston--Bennequin number is 
\[
  \L^{\tbb}_{(p_0,v_0)}(K,(M,\xi)) = \L^{\tbb}(K,(M,\xi))\cap \K_{(p_0,v_0)}(K,M).
\] 
To refer to link embedding spaces we will follow the same convention: if $L$ is a smooth link type, we will write $\L^{\tbb}(L,(M,\xi))$ and $\K(L,M)$ for the corresponding embedding spaces. We note that our links are always parametrized (ordered and oriented). 

Notice that the spaces of smooth embeddings and long smooth embeddings are related by a natural fibration induced by an evaluation map 
\begin{equation}\label{eq:LongEmbeddingsVSEmbeddings}
  \K_{(p_0,v_0)}(K,M)\hookrightarrow \K (K,M)\rightarrow \NS(TM).
\end{equation}
In the same vein, for the Legendrian case there is a fibration 
\begin{equation}\label{eq:LongEmbeddingsVSEmbeddingsLegendrian}
  \L^{\tbb}_{(p_0,v_0)}(K,(M,\xi))\hookrightarrow \L^{\tbb}(K,(M,\xi))\rightarrow \NS(\xi).
\end{equation}

The standard contact structure on $\NS^3\subseteq\C^2$ is defined as the field of complex tangencies of the $3$-sphere. Therefore, when $(M,\xi)=(\NS^3,\xi_\std)$ the base space in (\ref{eq:LongEmbeddingsVSEmbeddingsLegendrian}) is the Lie group $\U(2)$ which acts by contactomorphism on $(\NS^3,\xi_\std)$. It follows that there is a homotopy equivalence 
\begin{equation}\label{eq:LongHomotopyEquivalence}
  \L^{\tbb}(K)\cong \U(2)\times\L^{\tbb}_{(p_0,v_0)}(K).
\end{equation}

\subsubsection{Generalized Legendrian cables}\label{subsub:generalizedCables}
Let $K$ be a null-homologous knot in a $3$-manifold $M$, $\mu$ a meridian of $K$ and $\lambda$ a Seifert longitude. Take $N$ to be a neighborhood of $K$. For relatively prime integers $p$ and $q$, we define a $(p,q)$-cable $K_{p,q}$ of $K$ to be a knot on $\bd N$ with homology class $p[\lambda] + q[\mu] \in H_1(\bd N)$ and define the slope of $K_{p,q}$ to be $q/p$. If $q/p > \tbb(K) + 1$, we say $K_{p,q}$ is a \dfn{sufficiently positive cable} of $K$. In the same vein, we define the $(np,nq)$-cable link of $K$ for $n > 1$, denoted by $K_{np,nq}$ to be $n$ parallel copies of $K_{p,q}$ on $\bd N$. We also say that $K_{np,nq}$ is sufficiently positive if $K_{p,q}$ it is sufficiently positive. 

\begin{remark}
  In the broadest sense, we say $K_{p,q}$ is sufficiently positive if $p/q > \omega(K)$, where $\omega(K)$ is the {\it contact width} of $K$. For more details, see \cite{CEM:cables, EtnyreHonda:UTP}. 
\end{remark}

Now we will generalize the cabling operation to links. Given a knot $K\subseteq M$, we define $0\cdot K := \emptyset$, $1\cdot K := K$ and $K_{0,0} := \emptyset$. Consider an $n$ component link $L=(K^1,\ldots,K^n)$ such that each component is null-homologous. Also consider two finite sequences 
\begin{itemize}
  \item $B=\{i_1,\ldots,i_n\}\in\{0,1\}^n$ and 
  \item $C=\{(p_1,q_1),\ldots,(p_n,q_n)\}\in (\Z^2)^n$.
\end{itemize}
We define the \em (generalized) $(B,C)$-cable of $L$ \em to be the link 
\[ 
  L^B_C=(i_1 \cdot K^1,i_2\cdot K^2,\ldots, i_n\cdot K^n, K^1_{p_1,q_1},\ldots,K^n_{p_n,q_n}).
\] 
See Figure~\ref{fig:Seifert-Hopf} for example, which is a $((1,1),((0,0),(2,3)))$-cable of the Hopf link. We say that the link $L^B_C$ is \em sufficiently positive \em if either $K^j_{p_j,q_j}$ is sufficiently positive or $(p_j,q_j) = (0,0)$ for every $j\in\{1,\ldots,n\}$. Now let $L = (L^1, \ldots, L^n)$ be a Legendrian link. We define \em sufficiently positive (generalized) Legendrian $(B,C)$-cable of $L$ \em as follows.  
\[
  L^B_C=(i_1 \cdot L^1,i_2\cdot L^2,\ldots, i_n\cdot L^n, L^1_{p_1,q_1},\ldots,L^n_{p_n,q_n})
\]
where $L^j_{p_j,q_j}$ is a slope $q_j/p_j$ ruling curve(s) on a standard neighborhood of $L^j$.

If $L\subseteq (M,\xi)$ is a Legendrian link, we will denote by $\Op(L)$ an open tubular neighborhood of $L$ of unspecified size. We will always assume that $\Op(L)$ is small enough so that $(\Op(L),\xi|_{\Op(L)})$ is a standard neighborhood of $L$. The link complement of $L$ is denoted by $(C(L,M),\xi) := (M \setminus \Op(L), \xi_{M \setminus \Op(L)})$, a contact manifold with convex torus boundary. If $(M,\xi)$ is the standard $(\NS^3,\xi_\std)$ we will write $(C(L),\xi_\std)=(C(L,\NS^3),\xi_\std)$. Denote by $\mathcal{C}(C(L,M),\xi)$ the space of contact structures on $C(L,M)$ that are isotopic to $\xi$ and coincide with $\xi$ on $\Op(\partial C(L))$. 

\begin{theorem}\label{thm:SufficientlyPositiveLinkCables}
  Let $L$ be a Legendrian link in a tight contact $3$-manifold $(M,\xi)$ with the max-$\tb$ number in its link type, and $L^B_C$ a sufficiently positive Legendrian $(B,C)$-cable of $L$. Then the spaces of contact structures $\mathcal{C}(C(L,M),\xi)$ and $\mathcal{C}(C(L^B_C,M),\xi)$ are homotopy equivalent. 
\end{theorem}

The main technique for this theorem is Theorem \ref{thm:GluingAnnuli}, a parametric gluing result for tight contact $3$-manifolds along standard convex annuli that is of independent interest. To the best of our knowledge this is the first gluing type result along convex surfaces that are not disks with Legendrian boundary nor spheres. The non-parametric version of these cases was treated by Colin \cite{Colin:gluing} and the parametric version was treated by Mart\'inez-Aguinaga, Presas and the first author \cite{FMP:unknot}; and also by Muñoz-Ech\'aniz and the first author \cite{FM:ContactDehn}.

We discuss some applications of the previous result. Recall that a \em formal Legendrian embedding \em  of a link $L$ into a contact $3$-manifold $(M,\xi)$ is a pair $(\gamma,F_s)$  such that
\begin{itemize}
  \item $\gamma\in \K(L,M)$ and 
  \item $F_s:TL\hookrightarrow \gamma^* TM$, $s\in[0,1]$, is a homotopy of bundle monomorphisms such that $F_0=d\gamma$ and $\image(F_1)\subseteq \gamma^* \xi$.
\end{itemize}

We denote the space of formal Legendrian embeddings realizing a given smooth link type $L$ by $\FL(L,(M,\xi))$ and when $(M,\xi)=(\NS^3,\xi_\std)$ we simply write $\FL(L)$. If $\gamma$ is a Legendrian embedding realizing the smooth link $L$ we will write $\FL^\gamma(L)$ to denote the path connected component of the space of formal Legendrian embeddings that contains $\gamma$. Similarly, we will denote by $\L^\gamma(L)$ the path connected component of the space of Legendrian embeddings containing $\gamma$.

\begin{definition}
  A Legendrian link $L\subseteq (\NS^3,\xi_\std)$ satisfies the $C$-property if the space $\mathcal{C}(C(L),\xi_\std)$ is contractible. 
\end{definition}

It is well-known to experts that the connectivity of the space of tight contact structures in the complement of a Legendrian $L$ in $(\NS^3,\xi_\std)$ implies that every Legendrian $\widetilde{L}$ formally Legendrian isotopic\footnote{Two Legendrian links are formally isotopic if they have the same classical invariants \cite{FMP:Legendrian,Murphy:Loose}.} to $L$ is actually Legendrian isotopic to $L$. In particular, if the space $\mathcal{C}(C(L),\xi_\std)$ is contractible, so $L$ satisfies the $C$-property, then every sphere of Legendrian embeddings, based at some parametrization of $L$, that is contractible as a sphere of formal Legendrian embeddings is contractible (see Proposition \ref{prop:FormalInjectivity}). This observations make Theorem \ref{thm:SufficientlyPositiveLinkCables} a powerful tool to prove injective $h$-principles in a recursive way under the cabling operations:

\begin{corollary}\label{cor:InjectiveHPrinciple}
  Let $L\subseteq (\NS^3,\xi_\std)$ be a Legendrian link with the maximal Thurston--Bennequin number in its link type, and $L^B_C$ a sufficiently positive Legendrian cable of $L$. Then if $L$ satisfies the $C$-property, then $L^B_C$ also satisfy the $C$-property. In particular, if $\gamma$ parametrizes $L^B_C$ then the inclusion 
  \[
    \L^\gamma(L^B_C)\hookrightarrow \FL^\gamma(L^B_C) 
  \]
  is a homotopy injection.
\end{corollary}

\begin{example} There are several examples to which Corollary~\ref{cor:InjectiveHPrinciple} applies.
\begin{itemize}
  \item Every Legendrian unknot such that $\tb+|\rot|=-1$ satisfies the $C$-property, see \cite{FMP:unknot}.

  \item Every max-$\tb$ Legendrian positive torus link satisfies the $C$-property by applying Corollary~\ref{cor:InjectiveHPrinciple} to the max-$\tb$ Legendrian unknot. 

  \item More generally, let $f:\C^2\rightarrow \C$ be a complex polynomial with an isolated singularity at the origin. The \em link of the singularity \em $B_f=f^{-1}(0)\cap \NS^3(\varepsilon)$, for sufficiently small $\varepsilon>0$, is naturally transverse to $\xi_\std$ and is the binding of an open book decomposition supporting $\xi_\std$. However, as observed by Casals \cite{Casals:PlaneCurve}, there is a unique max-$\tb$ Legendrian representative $L_f$ of $B_f$ since the smooth link type of $B_f$ is a sufficiently positive iterated torus knot. It follows from Corollary~\ref{cor:InjectiveHPrinciple} that $L_f$ satisfies the $C$-property. 
\end{itemize}
\end{example}

\subsubsection{Legendrian cable knots}
Now we restrict our attention to cable knots. Let $(p,q)$ be a pair of relatively prime integers. Consider a standard neighborhood $N = J^1\NS^1$ of the max-$\tb$ Legendrian unknot in $(\NS^3, \xi_\std)$. Foliate $\bd N$ by slope $q/p$ ruling curves, the max-$\tb$ Legendrian $(p,q)$-torus knots. Then we can construct a loop of max-$\tb$ Legendrian embeddings $\gamma^\theta_{p,q}$, for $\theta\in\NS^1$, of $(p,q)$-torus knots by rotating $N$ about the core. This loop was introduced by \kalman \cite{Kalman:exotic}, which is known to be have an infinite order as a loop of formal Legendrians \cite{FMP:Legendrian}. We now generalize this construction to any Legendrian knot. Given a Legendrian knot $L\subseteq(M,\xi)$, we may define the associated \kalmans loop around $L$ to be 
\[
  \gamma^\theta_{L_{p,q}}=\varphi_L\circ \gamma^\theta_{p,q}
\] 
where $\varphi_L:(J^1\NS^1,\xi_\std)\rightarrow(\Op(L),\xi)$ is a contactomorphism sending the $0$-section to $L$. The reader may check that homotopy class of the loop is independent of the contactomorphism $\varphi_L$, see Lemma \ref{lem:FramingLegendrians}. 

Moreover, in \cite{FMP:satellite} it a continuous satellite map between Legendrian embedding spaces is defined. Since $(p,q)$-cables are a special instance of the satellite construction there is a continuous map 
\[ 
  C_{p,q}:\L^{\tbb}_{(p_0,v_0)}(K)\rightarrow \L^{\tbb}_{(p_0,v_0)}(K_{p,q}).
\]
This map sends, up to homotopy, every $k$-sphere of Legendrian embeddings in $\L^{\tbb}_{(p_0,v_0}(K)$ to the \em cable sphere \em in $\L^{\tbb}_{(p_0,v_0)}(K_{p,q})$: the one obtained by cabling all the embedding simultaneously. 

These two constructions explain the geometric meaning of the following result, which generalizes a result of Hatcher (see Budney \cite{Budney:knotsS3}) to the Legendrian setting.

\begin{theorem} \label{thm:cable}
  Let $K_{p,q}$ be a sufficiently positive $(p,q)$-cable of a knot $K$ in $\NS^3$. There is a homotopy equivalence
  \[ 
    \L^{\tbb}_{(p_0,v_0)}(K_{p,q})\cong \L^{\tbb}_{(p_0,v_0)}(K)\times \NS^1.
  \] 
  Moreover, let $L$ be a max-$\tb$ representative of $K$. Then a generator of the $\NS^1$-factor in the above homotopy equivalence is represented by a loop homotopic to $\gamma^\theta_{L_{p,q}}$.
\end{theorem}

We combine the previous result with the contractibility of the space of long Legendrian unknots with $\tb=-1$ in $(\NS^3,\xi_\std)$ \cite{FMP:unknot} and obtain the following corollary.

\begin{corollary}\label{cor:IteratedTorus}
  Let $K$ a sufficiently positive $n$-iterated torus knot contained in a ball $B \subset M$.
  \begin{itemize}
    \item [(i)] Let $(M,\xi)$ be a tight contact $3$-manifold. The natural inclusion $$ \L^{\tbb}_{(p_0,v_0)}(K,(M,\xi))\hookrightarrow\K_{(p_0,v_0)}(K,M)$$ is a homotopy equivalence.
    \item [(ii)] If $M = \NS^3$, there is a homotopy equivalence $$ \L^{\tbb}(K) \cong \U(2)\times (\NS^1)^n.$$  
  \end{itemize}
\end{corollary}

Note that the space of long smooth $n$-iterated torus knots in $\NS^3$ is homotopy equivalent to $(\NS^1)^n$, see \cite{Budney:knotsS3}. Therefore, in view of the homotopy equivalence (\ref{eq:LongHomotopyEquivalence}), the homotopy equivalence stated in Corollary~\ref{cor:IteratedTorus}.(ii) is a particular case of the one stated in Corollary~\ref{cor:IteratedTorus}.(i). In the proof, however, we will deduce (i) from the specialized case of (ii). To do so we will prove Proposition \ref{prop:FromS3ToM3} that is of independent interest and reduces the study of Legendrian embeddings in a general tight $3$-manifold that lie in a Darboux ball to the case of the standard $(\NS^3,\xi_\std)$.

\subsubsection{Legendrian Seifert fibered links}

A \dfn{Seifert fibered link} is a link in $S^3$ of which the complement is a Seifert fibered space. Let $(p,q)$ be a pair of relatively prime integers and $n$ a positive integer. There are four types of Seifert fibered links in $S^3$: a torus link $T_{np,nq}$, a Seifert link $S_{np,nq}$, a Seifert-Hopf link $R_{np,nq}$ and a keychain link $H_n$. See Figure~\ref{fig:torus},~\ref{fig:Seifert},~\ref{fig:Seifert-Hopf},~\ref{fig:keychain} and for more details, see Section~\ref{sec:lsflink}. In this paper, we only consider \em positive \em Seifert fibered links, meaning that $p,q > 0$ (we assume $H_n$ is vacuously positive). 

Legendrian positive Seifert fibered links are special cases of sufficiently positive generalized Legendrian cables. For the max-$\tb$ Legendrian positive Seifert fibered links we have a complete description of the contact structure on the complement, which is a Legendrian Seifert fibration over a disk. This allows us to completely describe the homotopy type of the embedding spaces. For every $n \geq 1$  we will denote by $\PB_n$ the pure braid group of $n$-strands. Recall that $K(G,1)$ denotes an Eilenberg--MacLane space.

\begin{theorem}\label{thm:main-SFlinkS3}
  Every max-$\tb$ Legendrian positive Seifert fibered link satisfies the $C$-property. Moreover, there are homotopy equivalences
  \begin{itemize}
    \item [(i)] $ \L^{\tbb}(T_{np,nq})\cong \U(2) \times K(\PB_{n+1}\times \Z^{n-1}, 1)$

    \item [(ii)] $ \L^{\tbb}(S_{np,nq})\cong \U(2) \times K(\PB_{n+1}\times \Z^n, 1) $

    \item [(iii)] $ \L^{\tbb}(R_{np,nq})\cong \U(2) \times K( \PB_{n+1}\times\Z^{n+1}, 1) $

    \item [(iv)] $ \L^{\tbb}(H_n)\cong \U(2) \times K( \PB_{n-1}\times \Z^{n-1}, 1 ) $
  \end{itemize}
\end{theorem}

\begin{remark} 
  The appearance of the pure braid groups in the previous statement is not a coincidence. The complement of all these links is a Legendrian Seifert fibration over the base space $\NS^2_{n,m}$: the $2$-sphere with $n$-holes and $m$ marked points that correspond to the singular fibers. Every element of the smooth mapping class group of the base $\pi_0(\Diff(\NS^2_{n,m}))\cong \PB_{n+m-1}\times \Z^{n-1}$ will give rise to a unique loop of Legendrian embeddings by producing loops of fibers combined with a possible reparametrization of the fibers. The content of our result is that, up to the action of the group of unitary matrices, every family of links can be understood in this way.
\end{remark}

As mentioned above, this results follows from the fact that the complement is a Legendrian Seifert fibered space. In particular, we determine the homotopy type of the group of contactomorphisms of a Legendrian Seifert fibration over a compact oriented surface with boundary that generalizes the result about Legendrian circle bundles \cite{FMP:unknot,Giroux:bundles} to the singular case.

\begin{theorem}\label{thm:Main-LSB}
  Let $(M,\xi)$ be a Legendrian Seifert fibration over a compact orientable surface $\Sigma^g_{n,m}$ with genus $g$, $n$ boundary components, $n>0$; and $m$ singular points. If the twisting number of regular fibers is $-1$, then we further assume the regular fiber is nontrivial in $\pi_1(M)$. Then, 
  \begin{itemize}
    \item [(i)] The space $\mathcal{C}(M,\xi)$ is contractible; and 
    \item [(ii)] The group $\Cont(M,\xi)$ is homotopy equivalent to $\pi = \pi_0(\Diff(\Sigma^g_{n,m}))$, the mapping class group of the base surface. In particular, it is homotopically discrete.
  \end{itemize}
\end{theorem}

\subsection*{Further questions}

One natural question regards the space of Legendrian embeddings that do not maximize the $\tb$ invariant. For instance, let $\gamma$ be an embedding of a Legendrian positive torus knot which does not maximize the $\tb$ invariant, but maximizes the self-linking number of its transverse push-off.

\begin{question}
    Assume that $\tb(\gamma)+|\rot(\gamma)|=pq-p-q$. Is the inclusion \[\L^\gamma_{(p_0,v_0)}(K_{p,q})\hookrightarrow \K_{(p_0,v_0)}(K_{p,q})\] a homotopy equivalence?
\end{question}

We expect the answer of the previous question to be \dfn{positive}. We point out that when $\tb(\gamma)+|\rot(\gamma)|<pq-p-q$ so $\gamma$ has at least one positive and one negative stabilization \cite{EtnyreHonda:torusKnots} the answer to the previous question is \dfn{negative} \cite{F:BypassesUnknots}.

Next, we can consider a Legendrian connected sum operation. Given two long Legendrian embeddings $\gamma_1$ and $\gamma_2$ there are homomorphisms 
\[ 
  \#_k:\pi_k(\L_{(p_0,v_0)}(\NS^3,\xi_\std),\gamma_1)\oplus\pi_k(\L_{(p_0,v_0)}(\NS^3,\xi_\std),\gamma_2)\rightarrow \pi_k(\L,\gamma_1\#\gamma_2), 
\]
induced by the connected sum map for \dfn{long} Legendrians. The study of this map when $k=0$, i.e. for single Legendrians, was done by Etnyre and Honda in \cite{EtnyreHonda:ConnectedSum}. For $k>0$ is it easy to see that $\#_k$ is \dfn{not} surjective in general. Indeed, the parametric Fuchs-Tabachnikov theorem (see \cite{FuchsTabachnikov,Murphy:Loose}) implies that any formal class can be represented by a genuine Legendrian class after taking connected sum with finitely many double stabilizations (pairs of positive and negative stabilizations). It would be interesting to study the required number of double stabilizations to generate every formal class. The previous observation together with the results in \cite{FMP:unknot} about Legendrian unknots (or any of the classification results provided in this article) implies that $\#_k$ cannot be surjective in general. However, the following question stills open and is crucial in understanding Question \ref{q:flexibility}: 

\begin{question}\label{q:ConnectedSum}
  Let $k>0$. Is the map $\#_k$ injective?
\end{question}

An affirmative answer of the injectivity question would answer Question \ref{q:flexibility} negatively for every component of the space of Legendrians, again by an application of the Fuchs-Tabachnikov theorem. Question \ref{q:ConnectedSum} can be essentially reduced to the study of families of $1$-standard annuli, with boundary two $\tb=-1$ Legendrian unknots (the lateral boundary of a $3$ dimensional contact $1$-handle). In this article we completely describe the space of $n$-standard annuli for $n>1$ when the boundary maximizes the twisting number. For $1$-standard annuli we are able to describe this space just under some additional hypotheses. The general case $n=1$ is more subtle and requires some new ideas. 

Related to the Fuchs-Tabachnikov result is the surjectivity version of Question \ref{q:flexibility}:
\begin{question}\label{q:Surjectivity}
  Let $\gamma$ be a Legendrian embedding. Is every $k$-sphere, $k>0$, of formal Legendrian embeddings, based at $\gamma$, homotopic to a sphere of genuine Legendrian embeddings?
\end{question}

This is the natural generalization of the Bennequin--Eliashberg inequality for Legendrians \cite{Bennequin,Eliashberg:20Years} which gives a negative answer in the case $k=0$ when the ambient contact manifold is tight. Although we do not explain it in detail here, we point out that the answer of this question is negative when $\gamma$ is any of the maximal Thurston-Bennequin Legendrian embeddings treated in the introduction as a consequence of our work.

Also, Question \ref{q:flexibility} is widely open for Legendrian knots in overtwisted contact manifolds. There are some partial results in both flexible/rigid directions at the $\pi_0$-level, see \cite{Dymara:overtwisted,EMM:nonloose,Vogel:overtwisted}, and also partial results pointing towards flexibility for higher homotopy groups \cite{CardonaPresas}. We can also ask the same question for Legendrian knots in high dimensional contact manifolds. There are answers of both flavors after the work of Murphy \cite{Murphy:Loose}, about loose Legendrians; the work (in progress) of Eliashberg and Kragh \cite{EliashbergKragh:ExoticLegendrians} about exotic families of standard Legendrian unknots in the standard contact $(2n+1)$-sphere for $n>1$; and the work of Sabloff and Sullivan \cite{SabloffSullivan:LoopsGenerating} about non-contractible families of Legendrians using generating functions.

We finish with two motivating constructions that can be done with families of Legendrian embeddings.  Both can also be carried out in a higher dimensional setting. 

\begin{construction}
  Let $\gamma^p$, $p\in\NS^k$, be a $k$-sphere of Legendrian embeddings in $(\R^3,\xi_\std)$ with fronts $\gamma^p_F$. Then, there is an associated Legendrian $L(\gamma^p):\NS^k\times\NS^1\rightarrow (\R^{2k+3},\xi_\std)$, defined by its front $(\gamma^p_F, p)$ in $J^0(\R\times\NS^k)\subseteq J^0(\R^{k+1})$. See \cite{EkholmKalman} in which the case $k=1$ is treated. The contact push-off \cite{CasalsEtnyre:ContactDivisor} of $L(\gamma^p)$ gives rise to a contact divisor embedding $C(\gamma^p):(\NS^k\times\NS^1\times\NS^k,\xi_\std)\rightarrow (\R^{2k+3},\xi_\std)$. This construction could give rise to new examples of interesting codimension $2$ contact submanifolds. The works \cite{CasalsEtnyre:ContactDivisor,CoteFauteux,EkholmKalman} could be useful to study such an embedding. 
\end{construction}

\begin{construction}
  The following construction was explained to the authors by Gay and involves parametrized surgeries in the spirit of \cite{Gay:S4}. Let $\gamma^\theta$ be a loop of Legendrians in $(M,\xi)$. Denote by $(M_L,\xi_L)$ the contact manifold obtained by Legendrian surgery on $L=\image(\gamma^0).$ Associated with the loop $\gamma^\theta$ there is a bundle 
  \[ 
    (M_L,\xi_L)\hookrightarrow X^4\rightarrow \NS^1,
  \] 
  where the fiber over $\theta\in\NS^1$ is the contact manifold obtained by Legendrian surgery over $\image(\gamma^\theta)$. The monodromy of this bundle is a contactomorphism $\varphi_{\gamma^\theta}\in\Cont(M_L,\xi_L),$ that we call the \em loop surgery contactomorphism. \em It is interesting to study the possibility of producing non-trivial but formally trivial contactomorphisms in this way. We point out that the same construction with $0$-surgeries gives rise to the contact Dehn-twist introduced in \cite{FM:ContactDehn} to produce the first examples of infinite order formally trivial but non-trivial contactomorphisms on reducible tight contact $3$-manifolds. We expect this construction to give rise to examples of \dfn{irreducible} tight contact $3$-manifolds with infinite order formally trivial but non-trivial contactomorphisms. 
\end{construction}

\subsection{Notation and conventions}
Throughout this paper every manifold will be assumed to be smooth, oriented and compact; and every contact structure will be assumed to be positively cooriented and tight. Throughout the article the word ``fibration'' stands for ``Serre fibration'' and we will usually say ``homotopy equivalence'' to refer to a ``weak homotopy equivalence''. The latter is not important since every functional space appearing in this article is equipped with the $C^\infty$-topology and the Whitehead theorem applies \cite{Milnor,Palais}. Here, we can easily extend all notation for links, except for the space of long embeddings.

Let $(M,\xi)$ be a contact $3$-manifold.
\begin{itemize}
  \item $\Diff(M)$: the group of orientation preserving diffeomorphisms of $M$. To shorten the notation, if $M$ has boundary, we always assume that the elements of $\Diff(M)$ are the identity near the boundary, i.e. $\Diff(M) = \Diff_+(M,\bd M)$.

  \item $\Cont(M,\xi)$: the group of coorientation preserving contactomorphisms of $(M,\xi)$. To shorten the notation, if $M$ has boundary, we always assume that the elements of $\Cont(M,\xi)$ are the identity near the boundary, i.e. $\Cont(M,\xi) = \Cont(M,\bd M,\xi)$.

  \item $\mathcal{C}(M,\xi)$: the space of contact structures on $M$ that are isotopic to $\xi$ and coincide with $\xi$ on $\Op(\partial M)$.

  \item $\K (K,M)$: the space of embeddings realizing the smooth knot type $K$ in $M$. 

  \item $\K (K) = \K(K,\NS^3)$: in case $K$ is a knot type in $\NS^3$. 

  \item $\K_{(p_0,v_0)}(K,M)=\{\gamma\in\K(K,M):(\gamma(0),\gamma'(0))=(p_0,v_0)\}$: the space of long smooth embeddings realizing the knot type $K$.

  \item $\L^{\tbb}(K,(M,\xi))$: the space of Legendrian embeddings realizing the smooth knot type $K$ with the \dfn{maximal} Thurston--Bennequin invariant. 

  \item $\L^{\tbb}(K) := \L^{\tbb}(K,(\NS^3,\xi_\std))$: in case $K$ is a knot type in $\NS^3$.

  \item $\L^{\tbb}_{(p_0,v_0)}(K,(M,\xi)) = \L^{\tbb}(K,(M,\xi))\cap \K_{(p_0,v_0)}(K,M)$: the space of long Legendrian embeddings realizing the knot type $K$ with the maximal Thurston--Bennequin invariant.

  \item $\L^{\gamma}(K,(M,\xi))$: the space of Legendrian embeddings that realize the knot type K and are Legendrian isotopic to $\gamma$ where $\gamma \colon \NS^1 \to (M, \xi)$ is a Legendrian embedding.

  \item $C(K,M) := M \setminus \Op(K)$: the complement of an open neighborhood of $K$.  

  \item $C(K) := C(K,\NS^3)$: in case $K$ is a knot type in $\NS^3$.  
  
  \item $(C(K,M),\xi)$: the contact structure on $C(K,M)$ that is the restriction of $(M,\xi)$ to $M \setminus \Op(K)$. Here, we assume $K$ is a Legendrian knot and $\Op(K)$ is a standard neighborhood of $K$.  

  \item $\NS(\xi)$: the unit sphere bundle of $\xi$.
\end{itemize}

\subsection*{Outline}
In Section~\ref{sec:convex}, we first review basic properties of convex surfaces. Then we study the homotopy type of spaces of standard convex surfaces and prove Theorem~\ref{thm:GluingAnnuli}, the gluing theorem along standard convex annuli. In Section~\ref{sec:LegendrianEmbeddings}, we study various properties of Legendrian embedding spaces. In Section~\ref{sec:cables}, we study the max-$\tb$ Legendrian embedding spaces of (generalized) cables and prove Theorem~\ref{thm:SufficientlyPositiveLinkCables}, Theorem~\ref{thm:cable} and Corollary \ref{cor:IteratedTorus}. In Section~\ref{sec:lsfs}, we study basic properties of Legendrian Seifert fibered spaces and prove Theorem~\ref{thm:Main-LSB}. Finally, in Section~\ref{sec:lsflink} we use the results in Section~\ref{sec:lsfs} to study the max-$\tb$ Legendrian embedding spaces of positive Seifert fibered links and prove Theorem~\ref{thm:main-SFlinkS3}.

\subsection*{Acknowledgements}
The authors thank John Etnyre for a helpful discussion about Legendrian Seifert fibered spaces. We are also grateful to David Gay for explaining to us the construction of the loop surgery contactomorphism. Finally, we would like to acknowledge Peter Lambert-Cole for asking ``the right question at the right time'' when the first author was presenting a preliminary version of this work at the UGA Topology seminar just involving torus links.

\section{Convex surfaces}\label{sec:convex}

In this section, we will study useful properties of some ``standard'' convex surfaces with a certain characteristic foliation and their $I$-invariant neighborhoods that will be used in later sections. In particular, we will show that the space of these convex surfaces is homotopy equivalent to the space of smooth surfaces. In Section~\ref{subsec:convexBasic}, we review basics of convex surfaces (for more details, see \cite{Honda:classification1, Massot:convex}). In Section~\ref{subsec:microfibration}, we review the \dfn{microfibration trick} introduced in \cite{FMP:unknot}. In Section~\ref{subsec:handlebody}, we review the homotopy type of the contactomorphism group of standard contact handlebodies. Finally, in Section~\ref{subsec:stdSurfaces}, we review and prove properties of various standard convex surfaces; and we prove our gluing result (Theorem \ref{thm:GluingAnnuli}).

\subsection{Convex surface basics}\label{subsec:convexBasic}
Recall that a \dfn{contact vector field} in a contact 3-manifold $(M,\xi)$ is a vector field whose flow preserves the contact structure $\xi$. An embedded surface $\Sigma$ in $(M,\xi)$ is \dfn{convex} if there exists a contact vector field $X$ transverse to $\Sigma$. If $\Sigma$ is a surface with boundary, then we assume $\bd\Sigma$ to be Legendrian. According to \dfn{Giroux's flexibility theorem}, any embedded surface can be $C^\infty$-perturbed (rel boundary) to be convex.

Given a convex surface $\Sigma$, the \dfn{dividing set} $\Gamma_{\Sigma}$ is a set of points on $\Sigma$ where the contact vector field $X$ is tangent to $\xi$. Also, $\Gamma_{\Sigma}$ is an embedded multicurve on $\Sigma$. The dividing set divides $\Sigma$ into two regions: 
\[  
  \Sigma \setminus \Gamma_{\Sigma} = R_+ \cup R_-
\]
where $R_+ = \{p \in \Sigma : \alpha(X_p) > 0 \}$, $R_- = \{p \in \Sigma : \alpha(X_p) < 0 \}$ and $\alpha$ is a contact form of $\xi$. If $\Sigma$ has Legendrian boundary, then $\tb(\bd\Sigma) = -\frac12 |\Gamma_{\Sigma} \cap \bd\Sigma|$ and $\rot(\bd\Sigma) = \chi(R_+) - \chi(R_-)$.

Suppose $\Sigma$ is a convex surface transverse to a contact vector field $X$. Let $\phi_t$ be the flow of $X$. Since $\xi$ is invariant under translations in the $t \in \mathbb{R}$ direction, for a small $\epsilon > 0$ we call $\phi_{[-\epsilon,\epsilon]}(\Sigma)$ an \dfn{$I$-invariant neighborhood of $\Sigma$}. The contact structure on an $I$-invariant neighborhood of a convex surface is completely determined by its dividing set. In particular, we have the following tightness criterion.
  
\begin{theorem}[Giroux's criterion]\label{thm:criterion}
  Let $(M,\xi)$ be a contact $3$-manifold and $\Sigma$ be a compact convex surface (possibly with Legendrian boundary) in $(M,\xi)$. 
  \begin{itemize}
    \item Suppose $\Sigma = S^2$. An $I$-invariant neighborhood of $\Sigma$ is tight if and only if $|\Gamma_{\Sigma}|=1$. 
    \item Suppose $\Sigma \neq S^2$. An $I$-invariant neighborhood of $\Sigma$ is tight if and only if there are no contractible dividing curves on $\Sigma$.
  \end{itemize}
\end{theorem}

A properly embedded graph $G$ on a convex surface is \dfn{non-isolating} if $G$ transversely intersects the dividing curves and each component of $\Sigma \setminus G$ intersects the dividing curves. Honda \cite{Honda:classification1} showed that any non-isolating graph can be $C^0$-perturbed into a Legendrian graph. 

\begin{theorem}[Legendrian realization principle, Honda \cite{Honda:classification1}]
  Suppose $G$ is a non-isolating properly embedded graph on a convex surface $\Sigma$ with the dividing set $\Gamma$. Then there exists a $C^0$-small isotopy of $\Sigma$ (rel boundary) such that $G$ is a part of the characteristic foliation. If $G$ is a simple closed curve, then the twisting number of $G$ with respect to the framing induced by $\Sigma$ is 
  \[
    \tw(G,\Sigma) = -\frac{|G \cap \Gamma|}{2}.
  \]
\end{theorem} 

A \dfn{bypass} is a half of an overtwisted disk constructed as follows. Consider a convex overtwisted disk whose dividing set consists of a single contractible closed curve. Take a properly embedded arc $\gamma$ on the disk intersecting the dividing curve in two points. By applying the Legendrian realization principle, we can assume that $\gamma$ is a Legendrian arc, and cut the disk along $\gamma$; each half-disk is called a bypass. By attaching a bypass, we can modify the dividing set on a convex surface. Suppose a bypass $D$ transversely intersects a convex surface $\Sigma$ such that $D \cap \Sigma = \gamma$. Let $\Gamma_\Sigma$ be the dividing set of $\Sigma$. Since the dividing curves interleave along the corner, $\gamma$ intersects $\Gamma_\Sigma$ in three points. We call the Legendrian arc $\gamma$ on $\Sigma$ the \dfn{attaching arc} of the bypass $D$ and say $D$ is a bypass for $\Sigma$. After edge-rounding, the convex boundary of a neighborhood of $D \cup \Sigma$ is a surface isotopic to $\Sigma$ but with its dividing set changed in a neighborhood of the attaching arc as shown in Figure~\ref{fig:bypass-attachment}. We call this process a \dfn{bypass attachment along $\gamma$}. Note that Figure~\ref{fig:bypass-attachment} is drawn for the case that the bypass $D$ is attached ``from the front'', that is, sitting above the page. If we attach a bypass ``from the back'' of $\Sigma$, the result will be the mirror image of Figure~\ref{fig:bypass-attachment}.

\begin{figure}[htbp]{\scriptsize
  \begin{overpic}[tics=20]{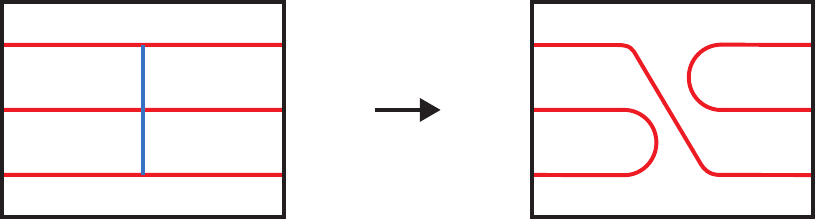}
  \end{overpic}}
  \vspace{0.2cm}
  \caption{Bypass attachment}
  \label{fig:bypass-attachment}
\end{figure}

Let $\Sigma$ be a convex surface and $D$ be a bypass for $\Sigma$. Suppose the attaching arc of $D$ passes three dividing curves $d_1$, $d_2$ and $d_3$ consecutively. We say the bypass $D$ is \dfn{effective} if $d_2$ is different from $d_1$ and $d_3$.  

Suppose $D$ is a non-effective bypass for a convex surface $\Sigma$ and let $\Sigma'$ be the resulting convex surface after attaching the bypass $D$ to $\Sigma$. Let $\Gamma_{\Sigma}$ be the dividing set on $\Sigma$. There are three types of non-effective bypasses for $\Sigma$ according to the effect on the dividing set:

\begin{enumerate}
  \item $\Gamma_{\Sigma'} = \Gamma_\Sigma$,
  \item $\Gamma_{\Sigma'}$ contains a contractible closed curve and $|\Gamma_{\Sigma'}| > |\Gamma_{\Sigma}|$,
  \item $\Gamma_{\Sigma'} \neq \Gamma_\Sigma$ and $|\Gamma_{\Sigma'}| \geq |\Gamma_{\Sigma}|$.
\end{enumerate} 

See Figure~\ref{fig:non-effective} for the first two cases. By Giroux's criterion (Theorem~\ref{thm:criterion}), the second type of bypasses does not occur in a tight contact manifold. The first type of bypasses does not change the dividing set, so we call it a \dfn{trivial bypass}. Honda \cite{Honda:bypass} showed that a trivial bypass is indeed trivial. 

\begin{lemma}[Honda \cite{Honda:bypass}]\label{lem:trivial}
  Suppose $\Sigma$ is a convex surface which is closed or compact with Legendrian boundary. If $D$ is a trivial bypass for $\Sigma$, then a neighborhood $N(\Sigma \cup D)$, which is a result of the bypass attachment, is an $I$-invariant neighborhood of $\Sigma$. 
\end{lemma}

\begin{figure}[htbp]
  \vspace{0.2cm}
  \begin{overpic}[tics=20]{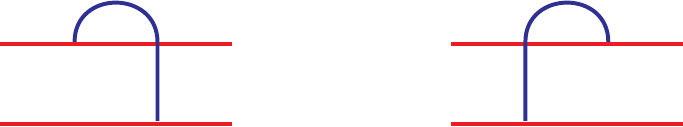}
  \end{overpic}
  \vspace{0.1cm}
  \caption{Two attaching arcs of non-effective bypasses.}
  \label{fig:non-effective}
\end{figure}

Suppose $\Sigma$ is a compact surface and $\xi$ is a contact structure on $\Sigma \times [0,1]$. Giroux \cite{Giroux:parametric}  showed that $\xi$ can be perturbed so that $\Sigma \times \{t\}$ are convex for all but finite $t \in [0,1]$. After that, Honda and Huang \cite{HH:parametric} extended the result to high dimensions. Colin \cite{Colin:discretization}  improved this result for a one-parameter family of embedded surfaces.

\begin{theorem}[Isotopy discretization, Colin \cite{Colin:discretization}, see also Honda \cite{Honda:bypass}]\label{thm:discretization}
  Let $(M,\xi)$ be a contact $3$--manifold and $\Sigma$ be a compact surface (possibly with boundary). Suppose $\phi_t\colon\Sigma \to M$ for $t\in[0,1]$ is a smooth isotopy of $\Sigma$ (rel boundary) and $\phi_0(\Sigma)$ and $\phi_1(\Sigma)$ are convex (possibly with Legendrian boundary). Then after a $C^{\infty}$-small perturbation of $\phi_t$ (rel boundary) while keeping $\phi_0$ and $\phi_1$ fixed, there exists a finite sequence $0 = t_0 < t_1 < \cdots < t_n = 1$ such that 
  \begin{itemize}
    \item $\phi_{t_i}(\Sigma)$ is convex for $i = 0, \ldots, n$.
    \item $\phi_{[t_i,t_{i+1}]}(\Sigma)$ is contactomorphic to a bypass attachment.
  \end{itemize}   
\end{theorem}

\subsection{The microfibration trick}\label{subsec:microfibration}
In this subsection, we review the \dfn{microfibration trick} introduced in \cite{FMP:unknot}, which is a key tool for studying the homotopy type of a space of convex surfaces. Suppose $M$ and $N$ are compact manifolds and $M$ is equipped with a metric $d$. Let $\Emb(N,M)$ be the space of smooth embeddings of $N$ into $M$. If $\bd N\neq \emptyset$, then we assume all the embeddings are fixed at $\Op(\partial N)$. For a given $\varepsilon>0$ and an embedding $e\colon N\rightarrow M$, denote by 
\[
  \mathcal{U}_{\varepsilon}(e)=\{f\in\Emb(N,M) : d(f(p),e(N))<\varepsilon \;\, \forall p\in N\}.
\]

\begin{theorem}[The microfibration trick, Fern\'andez--Mart\'inez-Aguinaga--Presas \cite{FMP:unknot}]\label{thm:microfibration}
Let $E_0\subseteq E_1\subseteq \Emb(N,M)$ be two subsets satisfying 

\begin{itemize}
  \item \textbf{Density Property:} The inclusion $j\colon E_0\hookrightarrow E_1$ is $C^0$-dense. 
  \item \textbf{Local Equivalence Property:} For every embedding $e\in E_1$, there exists a positive number $\varepsilon(e)>0$ such that for every $0<\varepsilon<\varepsilon(e)$ the inclusion 
  \[ 
    E_0\cap \mathcal{U}_{\varepsilon}(e)\hookrightarrow E_1\cap \mathcal{U}_{\varepsilon}(e)
  \] 
  is a weak homotopy equivalence.
\end{itemize}

Then $j$ is also a weak homotopy equivalence.
\end{theorem}

\begin{remark}
  The conclusion in the original statement in \cite{FMP:unknot} is stronger and provides a $C^0$-closedness property. However, we will not make use of it. 
\end{remark}

\subsection{Standard contact handlebodies}\label{subsec:handlebody} 
Let $(H_g,\xi)$ be a genus $g$ tight contact handlebody with convex boundary. We say $(H_g,\xi)$ is a \dfn{standard contact handlebody} if $g=0$ and it is a Darboux ball, or $g>0$ and it is a standard neighborhood of a Legendrian graph. This definition is equivalent to the requirement that $(H_g,\xi)$ is tight and every compressing disk in $(H_g,\xi)$ can be perturbed so that it intersects the dividing curves of $\partial H_g$ at exactly two points. The following result was proved in \cite{FMP:unknot} building on the main result of Eliashberg-Mishachev \cite{EliashbergMishachev:tightS3} which covers the case of the $3$-disk.

\begin{theorem}[Fern\'andez--Mart\'inez-Aguinga--Presas \cite{FMP:unknot}]\label{thm:StandardTightHandlebody}
  Let $(H_g,\xi)$ be a standard contact handlebody. Then the space $\mathcal{C}(H_g,\xi)$ is contractible and the inclusion 
  \[
    \Cont(H_g,\xi) \hookrightarrow \Diff(H_g)
  \]
  is a homotopy equivalence. 
\end{theorem}

\begin{remark}
  It follows from the work of Hatcher \cite{Hatcher:fibering,Hatcher:Smale} that the group $\Diff(H_g)$ is contractible. Therefore $\Cont(H_g,\xi)$ is also contractible.
\end{remark}

\subsection{Standard convex surfaces}\label{subsec:stdSurfaces}
In this subsection, we introduce some standard convex surfaces, which are characterized by specific dividing sets and characteristic foliations. Then we apply the microfibration trick (Theorem~\ref{thm:microfibration}) to study the spaces of these surfaces.

\subsubsection{Standard spheres} A \dfn{standard convex sphere} in a contact $3$-manifold $(M,\xi)$ is an embedded $2$-sphere $\NS^2\subseteq (M,\xi)$ with the same characteristic foliation as the boundary of a Darboux ball $(\D^3,\xi_\std)$. Notice that the dividing set consists of a single closed curve. A \dfn{standard embedding} of a sphere is any embedding that parametrizes a standard convex sphere. We denote by $\Emb_\std(\NS^2,(M,\xi))$ the space of standard embeddings of a sphere into $(M,\xi)$ and by $\Emb(\NS^2,M)$ the space of smooth embeddings of a sphere into $M$. The following is a direct application of Theorem \ref{thm:microfibration}, proved in \cite{FMP:unknot}.

\begin{theorem}[Fern\'andez--Mart\'inez-Aguinaga--Presas \cite{FMP:unknot}]\label{thm:StandardSpheres}
  The homotopy fiber of the inclusion $$ \Emb_\std(\NS^2,(M,\xi))\hookrightarrow \Emb(\NS^2,M)$$ is homotopy equivalent to $\Omega\NS^2$.
\end{theorem}

\subsubsection{Standard tori}\label{subsub:StandardTori} Let $T^2$ be an embedded torus in a contact $3$-manifold $(M,\xi)$. We fix a coordinate system on $T^2$ so that the slope of homologically essential curves on $T^2$ is well defined. In case of $T^2$ being the boundary of a solid torus, we define the slope of a curve $p\lambda + q\mu$ to be $q/p$ where $\mu$ is a meridian and $\lambda$ is a preferred longitude of the torus. 

\begin{definition}
  A convex torus $T^2 \subset (M,\xi)$ is a \dfn{standard convex torus} if 
\begin{itemize}
  \item $\Gamma$ consists of two dividing curves of slope $p/q$, 
  \item $T^2$ is foliated by Legendrian curves of any slope different from $p/q$, called \dfn{Legendrian rulings} and
  \item there are two singular lines parallel to the dividing curves, called \dfn{Legendrian divides}. 
\end{itemize}
\end{definition}

\subsubsection{Standard annuli} Let $A$ be an embedded annulus in a contact $3$-manifold $(M,\xi)$. We will study the space of convex annulus where all dividing curves run from one boundary component of $A$ to another component.  

\begin{definition}
  An annulus $A = \NS^1 \times [0,1] \subset (M,\xi)$ is an \dfn{$n$-standard convex annulus} if  
  \begin{itemize}
    \item  $\Gamma$ consists of $2n$ dividing curves $\{\frac{k-1}{n}\pi\}\times[0,1]$, $k\in\{1,\ldots,2n\}$.
    \item The characteristic foliation $\mathcal{F}_{n}$ consists of 
    \begin{itemize}
      \item Legendrian rulings $\NS^1\times\{t\}$, $t\in[0,1]$ and  
      \item  $2n$ Legendrian divides $\{\frac{2k-1}{2n}\pi\}\times[0,1]$, $k\in\{1,\ldots,2n\}$.
    \end{itemize}
  \end{itemize}
\end{definition}

Let $j \colon A \rightarrow (M,\xi)$ be a smooth embedding of an annulus $A$ into $M$. We say that $j$ is an \dfn{$n$-standard embedding} if $j^*\xi = \mathcal{F}_{n}$. We denote by $\Emb_{n,j}(A,(M,\xi))$ the space of $n$-standard embeddings that coincide with $j$ over $\Op(\partial A)$. Similarly, we denote by $\Emb_{j}(A,M)$ the space of smooth embeddings that coincide with $j$ over $\Op(\partial A)$. We also denote by $\Emb_{j}^{0}(A,M)$ the path connected component of $\Emb_{j}(A,M)$ that contains $j$ and define $\Emb_{n,j}^{0}(A,(M,\xi)) := \Emb_{n,j}(A,(M,\xi))\cap\Emb_{j}^{0}(A,M)$.

Let $A$ be an $n$-standard convex annulus. We define the {\it holonomy} of $A$ to be the isotopy class of the dividing curves on $A$. See Figure~\ref{fig:holonomy} for  an example of two standard convex annuli with different holonomies. If we fix a parametrization of the annulus, we can also define the holonomy to be the slope of dividing curves. Notice that if a contactomorphism fixes a standard annulus, it should preserve the holonomy. Thus contactomorphisms are more rigid than diffeomorphisms in general.

\begin{proposition}\label{prop:rigidityContacto}
  Suppose a contactomorphism $f\colon (M,\xi) \to (M,\xi)$ satisfies $f|_{\bd A} = id$ and $f(A) = A$, i.e. $f$ fixes $\bd A$ pointwise and $A$ setwise. Then $f|_{\Op(A)}$ is contact isotopic to the identity. 
\end{proposition}

\begin{proof}
  First, since $f|_A$ is a diffeomorphism on $A$, it is smoothly isotopic to $\tau^n$ where $\tau$ is a Dehn twist on $A$. However, since a contactomorphism preserves the dividing curves (so preserves the holonomy), $n$ should be $0$. Thus $f|_A$ is smoothly isotopic to the identity. Then by Giroux's flexibility theorem, $f(A)$ is contact isotopic to $A$. Furthermore, since a germ of a contact structure on $A$ is determined by the characteristic foliation on $A$, we may assume that $f|_{\Op(A)}$ is contact isotopic to the identity where $\Op(A)$ is an $I$-invariant neighborhood of $A$ ({\it c.f.}~\cite[Lemma~2.7]{Min:cmcg}).
\end{proof}

The following two technical lemmas will be used to prove Theorem~\ref{thm:StandardAnnuli}. The first one concerns the density property in Theorem \ref{thm:microfibration} for standard convex annuli. 

\begin{lemma}\label{lem:DensityPropertyStdAnnuli}
  Let $(M,\xi)$ be a tight contact $3$-manifold and $j\colon A\hookrightarrow(M,\xi)$ an $n$-standard embedding for $n \in \mathbb{N}$. Suppose each component of $j(\partial A)$ has the maximal twisting number with respect to any given framing. When $n = 1$, assume additionally that $j(A)$ unwraps in some tight covering of $(M,\xi)$ and the boundary still has the maximal twisting number. Then the inclusion 
  \[
    \Emb_{n,j}^{0}(A,(M,\xi))\hookrightarrow \Emb_{j}^{0}(A,M)
  \] 
  is $C^{\infty}$-dense (and hence $C^0$-dense).
\end{lemma}

\begin{proof}
  Let $i \colon A \to M$ be a smooth embedding in $\Emb_j^0(A,M)$. We first claim that $i$ can be $C^{\infty}$-small perturbed to be an $n$-standard embedding. By Grioux's flexibility theorem, we can make a $C^{\infty}$-small perturbation of $i$ so that $i(A)$ becomes a convex annulus while keeping $\Op(i(\bd A))$ fixed. Since each component of $i(\bd A)$ has the maximal twisting number, there are no boundary parallel dividing curves on $i(A)$. Otherwise, we can find a bypass containing a boundary parallel dividing curve and decrease the twisting number using this bypass. This implies that every dividing curve runs from one boundary component of $A$ to another component. By applying Giroux's flexibility theorem again, we can arrange the characteristic foliation $i^*\xi$ to be $n$-standard.

  It remains to show that $i$ is isotopic to $j$ through the standard embeddings. We first deal with the case $n > 1$. Let $\phi_t \colon A \to M$ be a smooth isotopy rel boundary from $i$ to $j$. By Colin's isotopy discretization (Theorem~\ref{thm:discretization}), there is a sequence $0 = t_0 < \cdots < t_n = 1$ such that $\phi_{t_{i+1}}(A)$ is obtained by attaching a bypass to $\phi_{t_{i}}(A)$. We inductively show that $\phi_{t_{i}}(A)$ is contact isotopic to $\phi_{t_i}(A)$ rel boundary. Clearly this is true for $\phi_{t_0}(A)$. Now suppose the claim holds for $\phi_{t_i}(A)$. Recall that there are two types of bypasses: effective and non-effective bypasses. Since $n>1$, any effective bypass for $\phi_{t_i}(A)$ intersects three different dividing curves. This bypass yields boundary parallel dividing curves as shown in Figure~\ref{fig:effective}. By the same argument above, this contradicts the fact that $i(\bd A)$ has the maximal twisting number. It is straightforward to check there are only two types of non-effective bypasses for $\phi_{t_i}(A)$ as shown in Figure~\ref{fig:non-effective}. If this bypass produces a contractible curve, then this contradicts that $(M,\xi)$ is tight. Therefore, the only possible non-effective bypasses for $\phi_{t_i}(A)$ are trivial bypasses. Thus $\phi_{t_{i+1}}(A)$ is contact isotopic to $\phi_{t_i}(A)$ so it is contact isotopic to $i(A)$. This completes the inductive argument.
    
  Finally, from the inductive argument above and Lemma~\ref{lem:trivial}, $\phi_{[t_i,t_{i+1}]}(A)$ are $I$-invariant neighborhoods for $i = 0,...,n-1$. Thus we can build a contact isotopy $\psi_t\colon A \to (M,\xi)$ such that $\psi_t^*\xi$ are identical for all $t \in [0,1]$ and this completes the proof. 
  
  Now for $n = 1$, we can apply the same argument above but the only difference is that an effective bypass does not intersect three different dividing curves. However, since $A$ unwraps in some covering of $(M,\xi)$, the effective bypass lifts to a bypass that intersects three different dividing curves. Thus attaching this bypass gives the same contradiction that the boundary of the lifted annulus has the maximal twisting number. 
\end{proof}

\begin{figure}[htbp]{\scriptsize
  \begin{overpic}[tics=20]{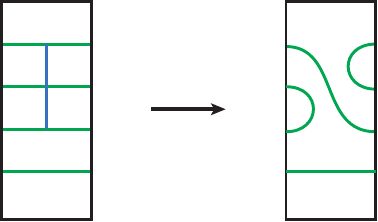}
  \end{overpic}}
  \vspace{0.2cm}
  \caption{An effective bypass attachment on an annulus}
  \label{fig:effective}
\end{figure}
  
\begin{remark}
  For Lemma~\ref{lem:DensityPropertyStdAnnuli}, the additional assumption when $n=1$ is essential. Without the assumption, there could be an effective bypass that intersects only two different dividing curves and it will change the holonomy of the dividing curves on $A$, see Figure~\ref{fig:holonomy}. This actually happens when we (de)stabilize a Legendrian knot. It will change the holonomy of an annulus by $1$, which is a part of a standard convex torus of the Legendrian knot. 
\end{remark}

\begin{figure}[htbp]{\scriptsize
  \begin{overpic}[tics=20]{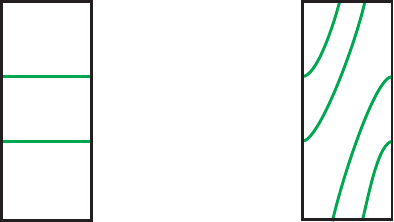}
  \end{overpic}}
  \vspace{0.2cm}
  \caption{Some possible holonomy of dividing curves on an annulus.}
  \label{fig:holonomy}
\end{figure}

The second lemma is about the local equivalence property in Theorem \ref{thm:microfibration} for standard convex annuli. Let $(U,\xi)$ be an $I$-invariant neighborhood of an $n$-standard annulus $A$. Notice that $(U = A \times [-1,1],\xi)$ is a standard solid torus after rounding the edges, and $A = A \times \{0\}$ separates $U$ into two solid tori $U_+ = U \times [0,1]$ and $U_- = U \times [-1,0]$. Consider an $n$-standard embedding $j \colon A \hookrightarrow (U,\xi)$. 

\begin{lemma}\label{lem:LocalPropertyStdAnnuli}
  The inclusion $\Emb_{n,j}(A,(U,\xi)) \hookrightarrow \Emb_{j}(A,U)$ is a homotopy equivalence. 
\end{lemma}

\begin{proof}
  Consider the following commutative diagram 
  \begin{displaymath} 
    \xymatrix@M=10pt{
    \Cont(U_{-},\xi)\times\Cont(U_+,\xi)\  \ar@{^{(}->}[d]\ar@{^{(}->}[r]  & \Cont(U,\xi) \ar[r] \ar@{^{(}->}[d] & \Emb_{n,j}(A,(U,\xi))   \ar@{^{(}->}[d] \\
    \Diff(U_{-})\times\Diff(U_{+}) \ar@{^{(}->}[r] & \Diff(U) \ar[r] &  \Emb_{j}(A,U) }
  \end{displaymath}
  in which the rows are fibrations. The statement follows from Theorem~\ref{thm:StandardTightHandlebody}.
\end{proof}

Now we are ready to apply Theorem~\ref{thm:microfibration} to the space of $n$-standard embeddings. 

\begin{theorem}\label{thm:StandardAnnuli}
  Let $(M,\xi)$ be a tight contact $3$-manifold and $j\colon A\hookrightarrow(M,\xi)$ an $n$-standard embedding for $n \geq 1$. Suppose each component of $\partial A$ has the maximal twisting number with respect to any given framing. When $n = 1$, assume additionally that $j(A)$ unwraps in some tight covering of $(M,\xi)$ and the boundary still has the maximal twisting number. Then, the inclusion 
  \[
    \Emb_{n,j}^{0}(A,(M,\xi))\hookrightarrow \Emb_{j}^{0}(A,M)
  \]  
  is a weak homotopy equivalence. 
\end{theorem}

\begin{proof}
  According to Lemmas~\ref{lem:DensityPropertyStdAnnuli} and~\ref{lem:LocalPropertyStdAnnuli}, the hypotheses of Theorem \ref{thm:microfibration} are fulfilled for $E_0=\Emb_{n,j}^{0}(A,(M,\xi))$ and $E_1=\Emb_{j}^{0}(A,M)$. Thus the result follows.
\end{proof}

Theorem~\ref{thm:StandardAnnuli} allows us to prove a parametric gluing result along $n$-standard annuli. This result will be the key in the rest of the article. First, we introduce some notation. Let $(M,\xi)$ be a (possibly disconnected) compact tight contact $3$-manifold with convex boundary. Assume that there are two disjoint embedded $n$-standard annuli $A_i\subseteq\partial M$ for $i=1,2$. Assume that $n\geq 2$ and $\partial A_i$ maximizes the twisting number with respect to any framing. From this data we obtain a new contact $3$-manifold $(M_A,\xi_A)$ by gluing $A_1$ and $A_2$ together. Here, we denote by $A\subseteq M_A$ the annulus obtained after identifying $A_1$ with $A_2$. 
We denote by $\Tight^{\twb}(M)$ the space of tight contact structures on $M$ that coincide with $\xi$ over $\Op(\partial M)$ and for which $\partial A_i$ maximizes the twisting number. Consider also the space of tight contact structures $\Tight^{\twb}(M_A)$ that is defined in the same way. 

\begin{theorem}\label{thm:GluingAnnuli}
  With the notation above, assume that $n\geq 2$ and $\partial A_i$ maximizes the twisting number in $(M,\xi)$. When $n = 1$, assume additionally that $j(A)$ unwraps in some tight covering of $(M,\xi)$ and the boundary still has the maximal twisting number. Then, the gluing map 
  \[
    \#_A \colon \Tight^{\twb}(M)\rightarrow \Tight^{\twb}(M_A),\quad\hat{\xi}\mapsto \#_A(\hat{\xi})=\hat{\xi}_A
  \]  
  is well-defined and is a homotopy equivalence.
\end{theorem}

\begin{proof}
  We first show that the map is well-defined by showing $\#_A(\hat{\xi})=\hat{\xi}_A$ is tight whenever $\hat{\xi}$ is tight and $\bd A$ has the maximal twisting number in $(M_A,\hat{\xi}_A)$. The second statement is not hard so we leave it as an exercise to readers. The first statement follows from Colin's Discretization technique (Theorem~\ref{thm:discretization}). Indeed, assume by contradiction that $\hat{\xi}_A$ is overtwisted. Then, there is an embedded overtwisted disk $\Delta\subseteq (M_A,\hat{\xi}_A)$ such that $A\cap \Delta\neq \emptyset$. Let $i:A\hookrightarrow (M_A,\hat{\xi}_A)$ be the inclusion of $A$, which is an $n$-standard embedding. Since $\Delta$ is a disk we may find an isotopy of smooth embeddings $e_t:A\hookrightarrow M_A$, $t\in[0,1]$, such that $e_0=i$ and $e_1(A)\cap \Delta=\emptyset$. We should see that $e_t$ can be chosen to be an isotopy of standard embeddings. Indeed, if this is the case then the contact Isotopy Extension Theorem implies the existence of a contactomorphism $\varphi:(M_A,\hat{\xi}_A)\rightarrow (M_A,\hat{\xi}_A)$ such that $\varphi\circ e_0=e_1$. Therefore, the overtwisted disk $\varphi^{-1}(\Delta)\cap A=\emptyset$ which is a contradiction. To find the isotopy through standard embeddings we proceed as follows. First, according to Theorem~\ref{thm:discretization}, after a possibly $C^\infty$-small deformation, we can find a sequence $0=t_0<t_1<\ldots<t_N=1$ such that $e_{[t_j,t_{j+1}]}(A)$ is contactomorphic to a bypass attachment to $e_{t_j}(A)$. As shown in the proof of Lemma~\ref{lem:DensityPropertyStdAnnuli}, there is no non-trivial bypass on $e_{t_0}(A)$, which is an $n$-standard annulus. Thus $e_{[t_0,t_1]}(A)$ is an $I$-invariant neighborhood so $e_{t_1}(A)$ is also an $n$-standard annulus. We keep applying the same argument until we show $e_{[t_{N-1},t_N]}$ is an $I$-invariant neighborhood. Therefore, $e_0(A)$ and $e_1(A)$ are isotopic through the standard embeddings. 

  Finally, we check that $\#_A$ is a homotopy equivalence. Let $K$ be a compact $CW$ complex and $G\subseteq K$ a subcomplex. We will see that if $\xi^k\in\Tight^{\twb}(M_A)$, $k\in K$, is a family of contact structures such that $\xi^k\in\#_A(\Tight^{\twb}(M))$, for $k\in G$, then we can find a homotopy $\xi^k_t\in\Tight^{\twb}(M_A)$, $t\in[0,1]$, such that 
  \begin{itemize}
    \item [(i)] $\xi^k_t=\xi^k$, for $(k,t)\in K\times \{0\}\cup G\times[0,1]$; and
    \item [(ii)] $\xi^k_1\in\#_A(\Tight^{\twb}(M))$, $k\in K$.
  \end{itemize}
  A tight contact structure $\xi\in \Tight^{\twb}(M_A)$ lies in the image $\#_A(\Tight^{\twb}(M))$ if and only if the inclusion $i:A\hookrightarrow M_A$ is $n$-standard. This motivates us to consider the space $\mathcal{E}$ conformed by pairs $(\xi,e_t)$ where $\xi\in \Tight^{\twb}(M_A)$ and $e_t:A\hookrightarrow M_A$, $t\in[0,1]$, is an isotopy of smooth embeddings between the inclusion $e_0=i$ and an $n$-standard embedding $e_1$ for the contact structure $\xi$. Gray Stability implies that there is a fibration 
  \[ 
    \mathcal{E}\rightarrow \Tight^{\twb}(M_A), (\xi,e_t)\mapsto \xi.
  \]
  Moreover, the fiber is contractible as a consequence of Theorem \ref{thm:StandardAnnuli}. This implies that if we understand the family $\xi^k$, $k\in K$, as a map $p:K\rightarrow \Tight^{\tbb}(M_A)$ then the pull-back fibration $$p^*\mathcal{E}\rightarrow K$$ has contractible fibers. In particular, the trivial section over the subcomplex $G\subseteq K$ defined by the constant isotopy $e^k_t=i$, $(k,t)\in G\times[0,1]$, can be extended over the whole base $K$. That is, we can find a family of smooth isotopies $e^k_t$, $(k,t)\in K\times[0,1]$, such that $e^k_t=i$, for $(k,t)\in K\times\{0\}\cup G\times[0,1]$; and  $e^k_1$ is a standard embedding for $\xi^k$. Apply the isotopy extension theorem to find $\varphi^k_t\in \Diff(M_A)$, $(k,t)\in K\times [0,1]$, such that $\varphi^k_t\circ e^k_0=e^k_t$, $(k,t)\in K\times [0,1]$; and $\varphi^k_t=\Id$ for $(k,t)\in K\times\{0\}\cup G\times[0,1]$. It follows that the homotopy of contact structures 
  \[
    \xi^k_t=(\varphi^k_t)^*\xi^k, (k,t)\in K\times[0,1],
  \]
  satisfies the required properties. This concludes the argument.
\end{proof}

\section{Legendrian embedding spaces}\label{sec:LegendrianEmbeddings}

In this section, we will study basic properties of Legendrian embedding spaces. In Section~\ref{subsec:LegendriansContactomorphisms}, we will review the relation between Legendrian embedding spaces and the group of contactomorphisms of the link complement. In Section \ref{subsec:LegendriansGeneralManifold}, we will prove a useful technical result about Legendrian embeddings that are contained in a Darboux ball in tight contact $3$-manifolds. In Section \ref{subsec:Pi2Invariant}, we review the $\pi_2$-invariant of loops of Legendrians that are contractible as loops of smooth embeddings introduced in \cite{FMP:Legendrian}. In particular, we will deduce that \kalmans loop has an infinite order as a loop of smooth \em long \em embeddings. Finally, in Section \ref{subsec:ContactStructuresComplement} we will prove a ``folk'' result about contact structures in the complement of a Legendrian knot in a parametric setup. Throughout this section we will use the notation $\L(K,(M,\xi))$ and $\L_{(p_0,v_0)}(K,(M,\xi))$ to refer to the space of Legendrian embeddings of a fixed knot type (without prescribing the $\tb$ invariant).

\subsection{Legendrians and contactomorphisms}\label{subsec:LegendriansContactomorphisms}
In this subsection we review the relation between the space of Legendrian embeddings $\L(L)$ and the group of contactomorphisms of the complement $(C(L),\xi_\std)$ for any given Legendrian link $L\subseteq(\NS^3,\xi_\std)$. First we review the relation in the case of long smooth knot embeddings. 

\subsubsection{Smooth case} Consider a knot $K\subseteq \NS^3$ and an embedding $\gamma\in\K_{(p_0,v_0)}(K)$ that parametrizes $K$. The main reduction trick used by Budney \cite{Budney:knotsS3} and Hatcher \cite{Hatcher:Knots1,Hatcher:Knots2} to study the homotopy type of $\K_{(p_0,v_0)}(K)$ is to apply the smooth isotopy extension theorem to obtain a fibration defined by post-composition over the embedding $\gamma$
\begin{equation}\label{eq:LongEmbeddingsFibration} 
  D\hookrightarrow \Diff(\D^3)\rightarrow \K_{(p_0,v_0)}(K).
\end{equation}
Here, we are identifying $\Diff(\D^3)$ with the subgroup of diffeomorphisms of $\NS^3$ that are the identity near $p_0$ and the fiber is $D=\{\varphi\in\Diff(\D^3):\varphi\circ\gamma=\gamma\}$.  However, as explained in \cite{Budney:knotsS3}, we have

\begin{lemma}\label{lem:FramingKnots}
    The inclusion $\Diff(C(K))\hookrightarrow D$ is a homotopy equivalence.
\end{lemma}

\begin{proof}
  Let $C$ be a compact CW complex and $G\subseteq C$ a subcomplex. Consider a family $\varphi^k\in D$, $k\in C$, such that $\varphi^k \in \Diff(C(K))$ for every $k\in G$. To prove the lemma, we will find a homotopy $\varphi^k_t\in D$, $t\in[0,1]$, such that $\varphi^k_0=\varphi^k$, for $k\in K$; $\varphi^k_1\in\Diff(C(K))$, for $k\in C$; and $\varphi^k_t=\varphi^k$, for $k\in G$. 

  Fix a trivialization of the normal bundle of $K$ and identify $\Op(K)=\NS^1\times \R^2 \subset \NS^3$. Since $\phi^k$ fixes $\gamma$, we may find $\varepsilon>0$ such that $\varphi^k(\NS^1\times \D^2(\varepsilon))\subseteq \NS^1\times\R^2$ and $\varphi^k_{|\NS^1\times\D^2(\varepsilon)} = i$ for $k\in G$ where $i$ is the inclusion. We denote the restrictions $\varphi^k_{|\NS^1\times\D^2(\varepsilon)}$ by $e^k$. Notice that $$e^k:\NS^1\times\D^2(\varepsilon)\hookrightarrow \NS^1\times\R^2$$ is a family of embeddings such that 
  \begin{itemize}
    \item [(i)] $e^k(\theta,0)=(\theta,0)$ for every $k\in C$, 
    \item [(ii)] $e^k=i$ is the inclusion for every $k\in G$ and
    \item [(iii)] $e^k_{|\Op(\{0\})\times\D^2(\varepsilon)} = i$ is also the inclusion for  $k\in C$. 
  \end{itemize}

  According to the Isotopy Extension Theorem, it is enough to check that the family of embeddings $e^k$ is contractible in the space of embeddings satisfying the above three conditions. To do so, we introduce the vertical dilation 
  \[
    f_t:\NS^1\times\R^2\rightarrow \NS^1\times\R^2, \;(\theta, v)\mapsto (\theta, tv).
  \] 
  Consider the family 
  \[
    e^k_t=f_t^{-1}\circ e^k\circ f_t, t\in[0,1]. 
  \] 
  Here, we define $e^k_0(\theta, v) := \lim_{t\mapsto 0^+}e^k_t(\theta,v)$, which gives a homotopy of embeddings between $e^k_1 = e^k$ and the vertical differential along the $0$-section of $e_k$. It is an easy exercise that this limit is well-defined and, in particular, $e^k_0$ satisfies 
  \[ 
    e^k_0(\theta,v)=(\theta,A^k_{\theta}(v))
  \]
  for some loop of matrices $A^k_{\theta}\in \Omega\GL^+(2,\R)$, $k\in C$. By the condition (iii), $A^k_{\theta}$ lies in the based loop space of $\GL^+(2,\R)$, instead of the free loop space. Since every diffeomorphism preserves the Seifert framing of a knot, $A^k_{\theta}$ lies in the component of $\Omega\GL^+(2,\R)$ that contains a constant loop $\Id$. Also, this component is contractible so we can homotope the family $A^k_\theta$ to the constant loop $\Id$ and this completes the proof.
\end{proof}

\begin{remark}\label{rmk:Framings}
  Observe that condition (iii) in the previous proof follows from the definition of long embeddings that we imposed at the beginning of the discussion. In particular, if we drop this condition or we work with multiple components links, the lemma is not true since non-trivial meridional rotations along the normal bundle could appear. However, we will see in Lemma \ref{lem:FramingLegendrians} that this is not the case for Legendrian embeddings. The reason is that a contactomorphism should preserve the contact plane so it could not rotate along the normal bundle of a Legendrian.
\end{remark}

Since $\Diff(\D^3)$ is contractible by Hatcher \cite{Hatcher:Smale}, we deduce from (\ref{eq:LongEmbeddingsFibration}) and Lemma \ref{lem:FramingKnots} the following:

\begin{lemma}\label{lem:KnotVSDiff}
  There is a homotopy equivalence $$\Omega\K_{(p_0,v_0)}(K)\cong\Diff(C(K)).$$ \qed
\end{lemma}

\subsubsection{Contact case} We follow the same scheme in the contact setting. Let $L\subseteq (\NS^3,\xi_\std)$ be a Legendrian link and $\gamma\in\L(L)$ a parametrization of the link. As in the knot case (\ref{eq:LongHomotopyEquivalence}) there is a homotopy equivalence
\begin{equation}\label{eq:HomotopyEquivalenceLongLinks}
  \L(L)\cong \U(2)\times \L_{(p_0,v_0)}(L)
\end{equation}
where $\L_{(p_0,v_0)}(L)=\{\gamma=(\gamma_1,\ldots,\gamma_n)\in\L(L): (\gamma_1(0),\gamma_1'(0))=(p_0,v_0) \}.$

As in the smooth case there is a fibration 
\begin{equation}\label{eq:LongLegendrianEmbeddingsFibration} 
  C\hookrightarrow  \Cont(\D^3,\xi_\std)\rightarrow \L_{(p_0,v_0)}(L)
\end{equation}
with fiber $C=\{\varphi\in\Cont(\D^3,\xi_\std):\varphi\circ\gamma=\gamma\}$. Again, we consider $\Cont(\D^3,\xi_\std)$ as the subgroup of contactomorphisms of $(\NS^3, \xi_\std)$ that are the identity near $p_0$. The main difference with the smooth case is that the following lemma works for multiple component links and regular (not long) Legendrian embeddings:

\begin{lemma}\label{lem:FramingLegendrians}
  The inclusion $\Cont(C(L),\xi_\std)\hookrightarrow C$ is a homotopy equivalence.
\end{lemma}

\begin{proof}
  A similar result appears in \cite{FMP:satellite}. However, we will provide a complete proof here for the sake of completeness. The proof is essentially the same as the one given in Lemma \ref{lem:FramingKnots}. The reader will see that, as mentioned in Remark \ref{rmk:Framings}, we will not appeal to condition (iii) in the proof of Lemma \ref{lem:FramingLegendrians}. Thus, our proof will readily apply for general Legendrian embeddings and Legendrian link embeddings. Therefore, we assume by simplicity that $L$ is connected, as the argument remains the same as in the non-connected case.  We take coordinates around $L$ given by $(J^1\NS^1=\NS^1\times\R^2, \xi_\std=\ker(dz-yd\theta))$. We use the same notation as in the proof of Lemma \ref{lem:FramingLegendrians}. The statement is reduced to proving that every family of contact embeddings 
  \[ 
    e_k: (\NS^1\times\D^2(\varepsilon),\xi_\std)\hookrightarrow (J^1\NS^1,\xi_\std),\; k\in K,
  \]
  satisfying conditions (i), (ii) and (iii) can be isotoped to the inclusion through embeddings that satisfy the same conditions. As mentioned above, condition (iii) will not be used. Note that the dilation $f_t:\NS^1\times\R^2\rightarrow\NS^1\times\R^2$,\; $t>0$, is a contactomorphism of $(J^1\NS^1,\xi_\std)$.  Hence, we apply the same argument as in Lemma \ref{lem:FramingKnots} in this case to homotope the family of contact embeddings $e^k$ into a new family of contact embeddings $e^k_0(\theta,v)=(\theta, A^k_\theta(v))$. Write $v=(y,z)$ and the entries of the matrix $A^k_\theta=(a^k_{i,j}(\theta))_{1\leq i,j\leq 2}$. The contact embedding condition reads as
  \[ 
    (e^k_0)^*(dz-yd\theta)=\lambda^k(dz-yd\theta) 
  \]
  for some family of smooth functions $\lambda^k:\NS^1\times\D^2(\varepsilon)\rightarrow \R_+$. From here we find that $a^k_{2,2}(\theta)=a^k_{1,1}(\theta)=\lambda^k(\theta,y,z)=\lambda^k(\theta)>0$, $a^k_{1,2}(\theta)=(\lambda^k)'(\theta)$ and $a^k_{2,1}(\theta)=0$. In particular, $\lambda^k(\theta,y,z)=\lambda^k(\theta)$ and the contact embedding $e^k_0$ has the form $$e^k_0(\theta,y,z)=(\theta,\lambda^k(\theta)y+(\lambda^k)'(\theta)z,\lambda^k(\theta)z).$$  Observe that there is further homotopy of contact embeddings 
  \[ 
    e^k_{0,s}(\theta,y,z)=(\theta, \lambda^k_s(\theta)y+(\lambda^k_s)'(\theta)z,\lambda^k_s(\theta)z), s\in[0,1],
  \] 
  where $\lambda^k_s=(1-s)\lambda^k+ s$; joining $e^k_{0,0}=e^k_0$ with the inclusion $e^k_{0,1}=i$. This concludes the argument. 
\end{proof}

Since the total space $\Cont(\D^3,\xi_\std)$ is contractible by Eliashberg--Mishachev \cite{EliashbergMishachev:tightS3}, we also conclude that 

\begin{lemma} \label{lem:LegVSCont}
  There is a homotopy equivalence $$\Omega\L_{(p_0,v_0)}(L)\cong\Cont(C(L),\xi_\std).$$ \qed
\end{lemma} 

\subsection{From Legendrians in \texorpdfstring{$(\D^3,\xi_\std)$}{S3} to Legendrians in \texorpdfstring{$(M,\xi)$}{M}.}\label{subsec:LegendriansGeneralManifold}
Let $(M,\xi)$ be a tight contact $3$-manifold. Let $e:(\D^3,\xi_\std)\hookrightarrow (M,\xi)$ be a contact  embedding of a Darboux ball and $\gamma\in\L(L,(\D^3,\xi_\std))$ a Legendrian embedding. We will regard $\gamma$ as a Legendrian in $(M,\xi)$ by making the identification $\gamma=e\circ \gamma\in \L(L,(M,\xi))$. Note that the isotopy class of $\gamma$ as a Legendrian in $\L(L,(M,\xi))$ does not depend on the contact embedding $e$ since the space of Darboux balls is connected. 

Recall that $\L^{\gamma}(L,(M,\xi))$ is the space of Legendrian embeddings that realize the smooth link type $L$ and are Legendrian isotopic to $\gamma$. The contact embedding $e:\D^3\hookrightarrow (M,\xi)$ induces continuous operators \[l_e:\L^\gamma(L,(\D^3,\xi_\std))\rightarrow\L^\gamma(L,(M,\xi)) \,\text{ and }\,  k_e:\K(L,\D^3)\hookrightarrow \K(L,M) \] that fit in a commutative diagram 
\begin{displaymath} 
  \xymatrix@M=10pt{
  \L^\gamma(L,(\D^3,\xi_\std)) \ar[r]^{l_e} \ar@{^{(}->}[d]^{i_0} & \L^\gamma(L,(M,\xi))  \ar@{^{(}->}[d]^{i_1} \\
  \K(L,\D^3) \ar[r]^{k_e} &  \K(L,M) }
\end{displaymath}
where $i_0$ and $i_1$ are the natural inclusions. In particular, $e$ induces a map 
\[
  j:\operatorname{Hofiber}(i_0)\rightarrow \operatorname{Hofiber}(i_1)
\]
between the homotopy fibers of these inclusions.

\begin{proposition}\label{prop:FromS3ToM3}
  The map $j:\operatorname{Hofiber}(i_0)\rightarrow \operatorname{Hofiber}(i_1)$ is a homotopy equivalence. 
\end{proposition}

\begin{remark}
  The same construction can be made for long Legendrians (in the connected case) and the induced map between the homotopy fibers is also a homotopy equivalence. The same applies for Legendrians and formal Legendrians. The proof of these facts are analogous to the one that we give below.
\end{remark}

\begin{proof}
  Note that we may identify $\operatorname{Hofiber}(i_0)=\{\gamma^u\in \K(L,\D^3), u\in[0,1]:\gamma^0=\gamma$ and $\gamma^1\in \L^\gamma(L,(\D^3,\xi_\std))\}$ and $\operatorname{Hofiber}(i_1)=\{\gamma^u\in \K(L,M), u\in[0,1]:\gamma^0=\gamma$ and $\gamma^1\in \L^\gamma(L,(M,\xi))\}$; so the map $j:\operatorname{Hofiber}(i_0)\rightarrow \operatorname{Hofiber}(i_1)$ is defined as $j(\gamma^u)=e\circ \gamma^u$.

  It is enough to prove that given a continuous family $\gamma^{z,u}\in \operatorname{Hofiber}(i_1)$, $z\in\D^k$, such that $\gamma^{z,u}\in j(\operatorname{Hofiber}(i_0))$ for $z\in\partial \D^k$; there exists a homotopy $\gamma^{z,u}_t\in\operatorname{Hofiber}(i_1)$, $(z,t)\in\D^k\times[0,1]$, such that 
  \begin{itemize}
    \item [(i)] $\gamma^{z,u}_1\in j(\operatorname{Hofiber}(i_1))$, $z\in\D^k$, and 
    \item [(ii)] $\gamma^{z,u}_t=\gamma^{z,u}$ for $(z,u,t)\in \D^k\times\{0\}\times[0,1]\cup \partial\D^k\times[0,1]\times[0,1]$. 
  \end{itemize}
  Here we are assuming that $\D^0$ is a single point so $\partial \D^0=\emptyset$. 

  By the hypothesis, $\image(\gamma^{z,u})\subseteq e(\D^3)$ for $(z,u)\in\D^k\times\{0\}\cup\partial \D^k\times[0,1]$. Therefore, by the smooth isotopy extension theorem we may find a family of smooth embeddings $e^{z,u}:\D^3\rightarrow M$, $(z,u)\in\D^k\times[0,1]$, such that 

  \begin{itemize}
    \item $\image(\gamma^{z,u})\subseteq e^{z,u}(\D^3)$ for every $(z,u)\in\D^k\times[0,1]$, 
    \item $e^{z,u}=e$ for $(z,u)\in\D^k\times\{0\}\cup\partial \D^k\times[0,1].$
  \end{itemize}

  Denote the restrictions of these balls to their boundary by $s^{z,u}=e^{z,u}_{|\partial\D^3}:\NS^2\rightarrow M$. Note that $s^{z,u}$ is a standard embedding for $(z,u)\in\D^k\times\{0\}\cup \partial\D^k\times[0,1]$. It follows from Theorem \ref{thm:StandardSpheres} (or by its proof given in \cite{FMP:unknot}) that there exists a homotopy of sphere embeddings $s^{z,u}_t:\NS^2\rightarrow M$, $(z,u,t)\in\D^k\times[0,1]\times[0,1]$, such that 

  \begin{itemize}
    \item $s^{z,u}_0=s^{z,u}$, for $(z,u)\in\D^k\times[0,1]$,
    \item $s^{z,u}_t=s^{z,u}=e_{|\partial\D^3}$, for $(z,u,t)\in\D^k\times\{0\}\times[0,1]\cup\partial\D^k\times[0,1]\times[0,1]$,
    \item $s^{z,u}_1$ is standard for $(z,u)\in\D^k\times[0,1]$; and 
    \item $\image(s^{z,u}_t)\subseteq \Op(\image(s^{z,u}_0))$ for $(z,u,t)\in\D^k\times[0,1]\times[0,1]$. In particular, $\image(s^{z,u}_t)\cap \image(\gamma^{z,u})=\emptyset$
  \end{itemize}

  In particular, it follows again from the smooth isotopy extension theorem that we may find a family of balls $\tilde{e}^{z,u}:\D^3\rightarrow M$, $(z,u)\in\D^k\times[0,1]$, such that 

  \begin{itemize}
    \item $\tilde{e}^{z,u}=e^{z,u}=e$, for $(z,u)\in\D^k\times\{0\}\cup\partial\D^k\times[0,1],$
    \item $\tilde{e}^{z,u}_{|\partial\D^3}=s^{z,u}_1$ is standard for $(z,u)\in\D^k\times[0,1]$, and 
    \item $\image(\gamma^{z,u}(\NS^1))\subseteq \tilde{e^{z,u}}(\D^3)$ for $(z,u)\in \D^k\times[0,1]$.
  \end{itemize}

  To conclude apply the contact isotopy extension theorem to the disk of standard sphere embeddings $s^{z,u}_1$ to find a family of contactomorphisms $\varphi^{z,u}\in\Cont(M,\xi)$, $(z,u)\in\D^k\times[0,1]$, such that 

  \begin{itemize}
    \item [(a)] $\varphi^{z,u}=\Id$, for $(z,u)\in\D^k\times\{0\}\cup \partial\D^k\times[0,1]$, and 
    \item [(b)] $\varphi^{z,u}\circ s^{z,0}=s^{z,u}$, for $(z,u)\in\D^k\times[0,1]$. 
  \end{itemize}

  In particular, it follows from condition (b) that $(\varphi^{z,u})^{-1}(\tilde{e}^{z,u}(\D^3))=e(\D^3)$ and, therefore, $(\varphi^{z,u})^{-1}\circ \gamma^{z,u}\in j(\operatorname{Hofiber}(i_0))$. We conclude that the homotopy 
  \[ 
    \gamma^{z,u}_t=(\varphi^{z,tu})^{-1}\circ \gamma^{z,u}, (z,u,t)\in\D^k\times[0,1]\times[0,1],
  \] 
  satisfies the required properties (i) and (ii) stated above (property (i) follows from (a)). This concludes the proof.
\end{proof}

\subsection{The \texorpdfstring{$\pi_2$}{pi2}-invariant for loops of (formal) Legendrians}\label{subsec:Pi2Invariant}
Here, we briefly review the $\pi_2$-invariant for loops of Legendrians introduced in \cite{FMP:Legendrian}.

\subsubsection{The $\pi_2$-invariant for loops of smooth long embeddings and torus knots} Let $\gamma^\theta\in \K_{(p_0,v_0)}(K)$, $\theta \in\NS^1$, be a loop of long smooth embeddings. Assume further that $\gamma^\theta$ is contractible as a loop inside  $\K(K)$. Then, $\gamma^\theta$ is contractible as a loop of long embeddings if and only if an obstruction class 
\begin{equation}\label{eq:Pi2Invariant} 
  \pi_2(\gamma^\theta)\in \Z_{m_K}
\end{equation}
vanishes. We call this obstruction class the $\pi_2$-invariant of the loop $\gamma^\theta$. 

Indeed, this follows by analyzing the long exact sequence in homotopy associated to the fibration (\ref{eq:LongEmbeddingsVSEmbeddings}). The relevant part for us is
\[ 
  \cdots \rightarrow \pi_2(\K(K))\rightarrow \pi_2( \NS(T\NS^3))\rightarrow \pi_1(\K_{(p_0,v_0)}(K))\rightarrow \pi_1(\K(K))\rightarrow\ldots 
\] 

Fix the global trivialization of $T\NS^3$ induced by the quaternion structure on $\R^4$. We have a well defined defined map 
\begin{equation}\label{eq:DerivativeMap} 
  D_0:\K(K)\rightarrow \NS^2, \gamma\mapsto \gamma'(0). 
\end{equation}
The non-negative integer $m_K$ is determined by the relation 
\[ 
  m_K\Z=\image \pi_2(D_0)<\pi_2(\NS^2)\cong\pi_2(\NS^3)\oplus\pi_2(\NS^2)\cong\pi_2(\NS(T\NS^3)).
\]

We now describe how to compute this obstruction class. Since $\gamma^\theta$ is contractible as a loop of smooth embeddings there exists a disk of embeddings $\gamma^{\theta,r}\in \K(K)$, $(\theta,r)\in\D^2$, bounding it, i.e. $\gamma^{\theta,1}=\gamma^\theta$. The map $D_0\gamma^{\theta,r}:\D^2\rightarrow \NS^2,(\theta,r)\mapsto D_0(\gamma^{\theta,r})$ descends to a map $D_0\gamma^{\theta,r}:\NS^2\rightarrow\NS^2$ because $\gamma^\theta$ is a loop of long smooth embeddings. The $\pi_2$-invariant of $\gamma^\theta$ is given by 
\[
  \pi_2(\gamma^\theta)=\deg(D_0\gamma^{\theta,r})\in\Z/m_K\Z.
\]
The only non unique, up to homotopy, choice in this process is the election of the capping disk of smooth embeddings. That is the reason that we should quotient out by $\image \pi_2(D_0)$ to obtain a homotopy invariant.

\subsubsection{The $\pi_2$-invariant for torus knots.} For torus knots it is easy to compute the integer $m_{T_{p,q}}$ from the work of Hatcher \cite{Hatcher:Knots1,Hatcher:Knots2}. See also \cite{Budney:knotsS3}.
\begin{theorem}[Hatcher]\label{thm:HatcherTorusKnots}
  For every pair of relatively prime integers $1<p<|q|$ the following holds:
  \begin{itemize}
    \item [(i)] The natural map $\SO(4)\rightarrow \K(T_{p,q}), \,A\mapsto A\circ \gamma;$ defined by postcomposition with a fixed embedding $\gamma$ is a homotopy equivalence.
    \item [(ii)] The $\pi_2$-invariant is $\Z$-valued, i.e. $m_{T_{p,q}}=0$.
    \item [(iii)] The map $\SO(2)\rightarrow \K_{(p_0,v_0)}(T_{p,q}),\,A_\theta\mapsto \gamma;$ defined by postcomposition of a fixed long embedding $\gamma$ with a rotation $A_\theta$, $\theta\in\NS^1$, along the $x$-axis in $\R^3$ is a homotopy equivalence. 
    \item [(iv)] The homomorphism $\Z\hookrightarrow \Diff(C(T_{np,nq})),\,k\mapsto \varphi^k$; where $\varphi$ is a meridional annular twist along the boundary of $C(T_{np,nq})$, is a homotopy equivalence.
  \end{itemize}  
\end{theorem}

The loop $A_\theta\gamma$, $\theta\in\NS^1$, of long embeddings described in (iii) is known as \em Gramain loop \em \cite{Gramain:Loop}. In view of the previous result this loop can be characterized as the unique loop $\beta^\theta\in\K_{(p_0,v_0)}(T_{p,q})$ such that 
\begin{equation}\label{eq:CharacterizationGramain}
  \pi_2(\beta^{2\theta})=1.
\end{equation}

\begin{remark}
  Gramain loop can be defined for every smooth long embedding $K$ and is a non-trivial element of the fundamental group of $\K_{(p_0,v_0)}(K)$ as long as $K$ is not the unknot (see \cite{Budney:knotsS3,Gramain:Loop}). The associated diffeomorphism in $\Diff(C(K))$ via the homotopy equivalence in Lemma \ref{lem:KnotVSDiff} is isotopic to a meridional annular twist along the boundary of $C(K)$. Note that a closed neighborhood of $\partial C(K)$ is $T^2\times[0,1]$ which has diffeomorphism group homotopy equivalent to $\Z\oplus \Z$, generated by meridional and longitudinal Dehn twists. See, for instance, ~\cite[Lemma~2.2]{HavensKoytcheff:knots}. When $K$ is Legendrian we may realize $(T^2\times[0,1])$ as a Legendrian fibration over an annulus $A$. In particular, the right handed Dehn twist in $A$ can be lifted to a contactomorphism $\Phi\in \Cont(T^2\times I,\xi)$  that we call a \em contact torus twist \em along $\partial C(K)$. This contactomorphism lies in the class $(1,1)\in\Z\oplus\Z\in\pi_0(\Diff(T^2\times I))$. It follows that $\pi_0(\Cont(C(K),\xi_\std))\rightarrow \pi_0(\Diff(C(K)))$ is surjective. The reader can check that in the case when $K$ is a torus knot the longitudinal Dehn twist is isotopic to the Identity in $\Diff(C(K))$.  In the next subsection we will see that \kalmans loop, understood as a loop of long Legendrians, is homotopic to Gramain loop (Corollary \ref{cor:SurjetivityLoopsTorusKnots}). To do this we just combine the previous Theorem with \cite{FMP:Legendrian}. An outcome of this article is that \kalmans loop is actually generated by a contact torus twist via Lemma \ref{lem:LegVSCont}. 
\end{remark}

\subsubsection{The $\pi_2$-invariant for loops of (formal) Legendrians} The previous discussion can be particularized to the case of Legendrian embeddings giving rise to an invariant of (formal) Legendrian embeddings. This invariant was previously treated in \cite{FMP:Legendrian} although it was introduced from a different perspective. 

Recall that in the Legendrian setting there is a homotopy equivalence 
\[
  \Phi:\L(K)\rightarrow \U(2)\times\L_{(p_0,v_0)}(K), \gamma\mapsto (A_\gamma, \widetilde{\gamma}=A_\gamma^{-1}\gamma). 
\]

\begin{definition}
  Let $\gamma^\theta\in\L(K)$, $\theta\in\NS^1$, the \em rotation number \em of the loop is 
  \[ 
    \rot_{\pi_1}(\gamma^\theta)=[A_{\gamma^\theta}]\in\pi_1(\U(2))\cong\Z
  \]
\end{definition}

Under the trivialization of $T\NS^3$ induced by the quaternions the standard contact structure trivializes as $\xi_\std(p)=\langle j\cdot p, k\cdot p\rangle$. Therefore, the derivative map (\ref{eq:DerivativeMap}) applied over a loop of Legendrians $\gamma^\theta$ can be understood as  $D_0\gamma^\theta:\NS^1\rightarrow \NS^1$. The rotation number of the loop can be computed as 
\[ 
  \rot_{\pi_1}(\gamma^\theta)=\deg(D_0\gamma^\theta).
\]

We see that for every loop of Legendrians $\gamma^\theta$ there is an associated loop of long Legendrians $\widetilde{\gamma}^\theta$. Moreover, both loops are homotopic if and only if $\rot_{\pi_1}(\gamma^\theta)=0$. 

\begin{example}\label{ex:KalmanLoopIsLong}
  Every \kalmans loop $\gamma^\theta_{L_{p,q}}$ around a Legendrian $L=\gamma(\NS^1)$ satisfies that 
  \[ 
    \rot_{\pi_1}(\gamma^\theta_{L_{p,q}})=0
  \]
  since $(\gamma^\theta_{L_{p,q}})'(0)$ is homotopic through Legendrian vectors to $\gamma'(0)$. In particular, these loops are homotopic through Legendrian loops to the loop of long Legendrian embeddings $\widetilde{\gamma}^\theta_{L_{p,q}}$. 
\end{example}

\begin{definition}
  Let $\gamma^\theta\in\L(K)$, $\theta\in\NS^1$, be a loop of Legendrian embeddings such that $\widetilde{\gamma}^\theta$ is contractible as a loop of smooth embeddings. The $\pi_2$-invariant of the loop $\gamma^\theta$ is defined as 
  \[ 
    \pi_2(\gamma^\theta)=\pi_2(\widetilde{\gamma}^\theta)\in\Z/m_K\Z.
  \]
\end{definition}

\begin{example}
  Let $\gamma^\theta_{p,q}\in \L^{\tbb}(T_{p,q})$, $\theta\in\NS^1$, be the loop of max-$\tb$ positive $(p,q)$-torus knots introduced by \kalman in \cite{Kalman:exotic}. This loop has the following properties: 
  \begin{itemize}
    \item [(i)] As explained above the rotation number of the loop is $\rot_{\pi_1}(\gamma^\theta_{p,q})=0.$ Therefore, $\gamma^\theta_{p,q}$ is homotopic through Legendrian loops to $\widetilde{\gamma}^\theta_{p,q}$. In an abuse of notation we will write $\gamma^\theta_{p,q}=\widetilde{\gamma}^\theta_{p,q}$ and assume beforehand that \kalmans loop is conformed by long knots.
    
    \item [(ii)] $\gamma^{2\theta}_{p,q}$, $\theta\in\NS^1$, is contractible as a loop of smooth embeddings. In particular, $\gamma^{2\theta}_{p,q}$, has a well-defined $\pi_2$-invariant.
   
    \item [(iii)] $\gamma^{2\theta}$ is non-contractible as a loop of (formal) Legendrian embeddings \cite{FMP:Legendrian,Kalman:exotic}. The non-triviality of the loop was proved by \kalman by using the monodromy invariant coming from LCH, introduced also by himself. The formal non-triviality was proved in \cite{FMP:Legendrian} by computing 
    $$ \pi_2(\gamma^{2\theta}_{p,q})=1\in\Z. $$
    In particular, $\gamma^{2\theta}_{p,q}$ is \dfn{not contractible} as a loop of long smooth embeddings. It also follows that this loop has infinite order (as a loop of long smooth embeddings and as a loop of formal Legendrians).
    
    \item [(iv)] In view of (\ref{eq:CharacterizationGramain}) the loop $\gamma^\theta_{p,q}$ is homotopic to Gramain loop.
  \end{itemize}
\end{example}

As a consequence of this discussion and Theorem \ref{thm:HatcherTorusKnots} we have

\begin{corollary}\label{cor:SurjetivityLoopsTorusKnots}
  The restriction of the inclusion 
  \[ 
    i:\{\gamma^{\theta}_{p,q}:\theta\in\NS^1\}\subseteq \L^{\tbb}_{(p_0,v_0)}(T_{p,q})\hookrightarrow \K_{(p_0,v_0)}(T_{p,q})
  \] 
  is a homotopy equivalence.
\end{corollary}

\subsection{Contact structures in the complement of a Legendrian}\label{subsec:ContactStructuresComplement}
The following is a well-known folk result in the non-parametric case. Recall that $\L^{\gamma}(L)$ is the space of Legendrian embeddings that realize the smooth link type $L$ in $\NS^3$ and are Legendrian isotopic to $\gamma$.

\begin{proposition}\label{prop:FormalInjectivity}
  Let $L\subseteq (\NS^3,\xi_\std)$ be a Legendrian link that satisfies the $C$-property and $\gamma$ an embedding that parametrizes L. Then, the inclusion $$ \L^\gamma(L)\hookrightarrow \FL^\gamma(L) $$ is a homotopy injection.
\end{proposition}

\begin{proof}
  Fix polar coordinates $(z,r)\in\D^k=\NS^{k-1}\times[0,1]/\sim$ on the disk. Let $$(\gamma^{z,r}, F^{z,r}_s)\in\pi_k(\FL^\gamma,\L^\gamma)$$ be a relative homotopy class, i.e. $(\gamma^{z,1},F^{z,1}_s)=(\gamma^{z,1},(\gamma^{z,1})')$ is Legendrian. To prove the proposition, we should show that $\gamma^{z,1}$ is a contractible family of Legendrian embeddings.

  Every Legendrian embedding $\gamma^{z,1}$ is framed by means of the framing $\langle i\cdot(\gamma^{z,1})', i\cdot\gamma^{z,1} \rangle$. Here, we consider $\xi_\std$ as the complex tangencies of $\NS^3 \subset \mathbb{R}^4$. This family of framings can be extended over the whole family $(\gamma^{z,r},F^{z,r}_s)$. Indeed, $\langle i\cdot F^{z,r}_1, i\cdot \gamma^{z,r} \rangle$ defines a framing for the normal bundle $\nu(F^{z,r}_1)$ and the formal Legendrian structure allows us to extend this framing over $\nu(F^{z,r}_s)$ for every $((z,r),s)\in\D^k\times[0,1]$. Denote by $\mathcal{F}^{z,r}$ the framings over $\nu(F^{z,r}_0)=\nu((\gamma^{z,r})')=\nu(\gamma^{z,r})$.

  Apply the smooth Isotopy Extension Theorem to find a family of diffeomorphisms $\varphi^{z,r}$, $(z,r)\in\D^k$, such that 
  \[ 
    \varphi^{z,r}\circ\gamma^{z,r}=\gamma^{N,1},
  \] 
  where $N\in\NS^{k-1}$ denotes the north pole. Since the parameter space is the disk $\D^k$ we may assume, maybe after a further isotopy, that the family of framings of $\nu(\gamma^{N,1})$ given by 
  \[ 
    \varphi^{z,r}_* \mathcal{F}^{z,r}\equiv \mathcal{F}^{N,1}
  \] 
  is constant. In particular, for $r=1$ we can assume the diffeomorphism $\varphi^{z,1}$ sends the contact framing $\mathcal{F}^{z,1}$ over the Legendrian $\gamma^{z,1}$ to the contact framing $\mathcal{F}^{N,1}$ over the Legendrian $\gamma^{N,1}$. It follows from Lemma \ref{lem:WeinsteinWithParameters}, to be proved below, that we may further isotope the family $\varphi^{z,1}$, relative to the condition $\varphi^{z,1}\circ\gamma^{z,1}=\gamma^{N,1}$, into a new family $\phi^z$ such that 
  \begin{itemize}
    \item $\phi^z\circ \gamma^{z,1}=\gamma^{N,1}$ 
    \item $\phi^{z}_{|\Op(\gamma^{z,1})}:(\Op(\gamma^{z,1}),\xi_\std)\rightarrow (\Op(\gamma^{N,1}),\xi_\std)$ is a contactomorphism.
  \end{itemize}

  Let $L=\image \gamma^{N,1}$. The family of tight contact structures $\xi^z=\phi^z_*\xi_\std$ lies in the space $\mathcal{C}(C(L),\xi_\std)$. Note that we may assume that $\xi^N=\xi_\std$. Since $L$ satisfies the $C$-property, this family is contractible so there exists a homotopy $\xi^z_t\in\mathcal{C}(C(L),\xi_\std)$, $t\in[0,1]$, such that $\xi^z_0=\xi^z$, $\xi^N_t=\xi_\std$ and $\xi^z_1=\xi_\std$. Thus we can apply the Gray stability theorem and find a family $G^z\in\Diff(C(L))$ such that $G^z\circ \xi^z=\xi_\std$ and $G^N=\Id$. Here we consider $G^z$ as diffeomorphisms of $\NS^3$ fixing $\Op(L)$ point-wise. 

  In particular, the family of diffeomorphisms $H^z=G^z\circ \phi^z$, $z\in\NS^{k-1}$, satisfies that 
  \begin{itemize}
    \item $H^z\circ\gamma^{z,1}=\gamma^{N,1}$ and 
    \item $H^z\in\Cont(\NS^3,\xi_\std)$.
  \end{itemize}

  This implies that the initial sphere of Legendrian embeddings $\gamma^{z,1}=(H^z)^{-1}\gamma^{N,1}$ is generated by a sphere of contactomorphisms. We claim that this sphere of contactomorphisms is contractible. Indeed, it follows from Eliashberg--Mishachev \cite{EliashbergMishachev:tightS3} that the evaluation map $\operatorname{ev}_{(p_0,v_0)}:\Cont(\NS^3,\xi_\std)\rightarrow \U(2), \;\Phi\mapsto (\Phi(p_0), d_{p_0} \Phi (v_0))$ is a homotopy equivalence. The homotopy class of the evaluation map does not depend on the choice of $(p_0,v_0)\in\NS(\xi_\std)$ so we fix $(p_0,v_0)=(\gamma^{N,1}(0),(\gamma^{N,1})'(0)))$ to define it. It follows that 
  \[ 
    \operatorname{ev}_{(p_0,v_0)}((H^z)^{-1})=(\gamma^{z,1}(0),(\gamma^{z,1})'(0))\in \U(2), z\in\NS^{k-1}
  \] 
  is a contractible sphere inside $\U(2)$ since it bounds the disk $(\gamma^{z,r}(0),F^{z,r}_1(0))$. Therefore, the sphere of contactomorphisms $(H^z)^{-1}$ is contractible as claimed and, therefore, the family of Legendrians $\gamma^{z,1}$ is contractible. This concludes the argument.
\end{proof}

It remains to prove the following technical result.

\begin{lemma}\label{lem:WeinsteinWithParameters}
  Let $(M,\xi)$ be a contact $3$-manifold. Let $K$ be a compact parameter space and $\gamma^k$, $k\in K$, be a family of Legendrian embeddings. Assume that there exists a Legendrian embedding $\gamma$ and a family of diffeomorphisms $\varphi^k\in \Diff(M)$, $k\in K$, such that
  \begin{itemize}
    \item $\varphi^k\circ \gamma^k=\gamma$, for $k\in K$; and 
    \item $(\varphi^k)_*(\xi_p)=\xi_{\varphi^k(p)}$ as oriented plane fields, for every $p\in\image(\gamma^k)$.
  \end{itemize}
  Then, there exists a homotopy $\varphi^k_t\in\Diff(M)$, $(k,t)\in K\times[0,1]$, such that 
  \begin{itemize}
    \item [(i)] $\varphi^k_0=\varphi^k$, $k\in K$;
    \item [(ii)] $\varphi^k_t \circ \gamma^k=\gamma$, $(k,t)\in K\times [0,1]$; and 
    \item [(iii)] $(\varphi^k_1)_{|\Op(\gamma^k)}$ is a contactomorphism for $k\in K$. 
  \end{itemize}
\end{lemma}

\begin{proof}
  Fix coordinates of $(J^1\NS^1,\xi_\std)$, a tubular neighborhood of $\gamma(\NS^1)$, in such a way that $\gamma(\theta)=(\theta,0,0)\in J^1\NS^1\subseteq M$. Here, the coordinates in $J^1\NS^1=\NS^1\times\R^2$ are $(\theta,y,z)$ and $\xi_\std=\ker(dz-yd\theta)$. 

  By a fat Legendrian embedding we mean a contact embedding of a standard neighborhood $e\colon (\NS^1\times\D^2,\xi_\std)\rightarrow (M,\xi)$. Note that the restriction of a fat Legendrian to the $0$-section (the core of $\NS^1\times\D^2$) is a genuine Legendrian embedding. In fact, this restriction map defines a homotopy equivalence between the space of fat Legendrian embeddings and the space of Legendrian embeddings (this is proved in \cite{FMP:satellite}, although the proof is essentially given in Lemma \ref{lem:FramingLegendrians}). If we combine this with the compactness of the families $\gamma^k$ and $\varphi^k$ we may find a family of contact embeddings $e^k:(\NS^1\times\D^2,\xi_\std)\rightarrow (M,\xi)$, $k\in K$, such that 
  \begin{itemize}
    \item [(i)] $e^k(\theta,0,0)=\gamma^k(\theta)$, for $k\in K$;
    \item [(ii)] $\tilde{e}^k=\varphi^k \circ e^k: \NS^1\times \D^2\rightarrow  J^1\NS^1\subseteq M$ is a smooth embedding, for $k \in K$;
    \item [(iii)] $\tilde{e}^k(\theta,0,0)=\gamma(\theta)=(\theta,0,0)\in J^1\NS^1$, for $k\in K$ and;
    \item [(iv)] $d_{(\theta,0,0)}\tilde{e}^k(\xi_\std)=\xi_{\std,(\theta,0,0)}$ as oriented plane fields, for $k\in K$.
  \end{itemize}

  From now on we will focus on the family of smooth embeddings $$\tilde{e}^k:\NS^1\times \D^2\rightarrow (J^1\NS^1,\xi_\std).$$ By the smooth isotopy extension theorem the result will follow if we prove that there exists a homotopy of embeddings $$\tilde{e}^k_t:\NS^1\times\D^2\rightarrow J^1\NS^1,(k,t)\in K\times[0,2],$$ satisfying properties (i), (ii) and (iii) above and such that $\tilde{e}^k_0=\tilde{e}^k$ and $\tilde{e}^k_2:(\NS^1\times\D^2,\xi_\std)\rightarrow (J^1\NS^1,\xi_\std)$ is a contact embedding. We built such a homotopy in two steps. 

  First, consider the contact dilation along the fiber directions $$f_t:(J^1\NS^1,\xi_\std)\rightarrow (J^1\NS^1,\xi_\std),(\theta,y,z)\mapsto (\theta, ty,tz)$$ which is a contactomorphism for $t>0$. Define the first half of the homotopy as 
  \[
    \tilde{e}^k_t=f_{1-t}^{-1}\circ \tilde{e}^k\circ f_{1-t}:\NS^1\times\D^2\rightarrow J^1\NS^1, (k,t)\in K\times[0,1]. 
  \] 

  For $t=1$ one should understand the previous expression as $$\tilde{e}^k_1(\theta,y,z)=\lim_{t\mapsto 1^-} f_{1-t}^{-1}\circ \tilde{e}^k\circ f_{1-t}(\theta,y,z)$$ which is just the vertical differential of $\tilde{e}^k$ along the $0$-section. In particular, $\tilde{e}^k_1$ is linear in the fiber directions so it has the form 

  \[ 
    \tilde{e}^k_1(\theta, y,z)=(\theta, a_{11}^{k}(\theta)y+a_{12}^{k}(\theta)z,a_{21}^{k}(\theta)y+a_{22}^{k}(\theta)z), 
  \] 
  for some family of smooth functions $a_{i,j}^{k}:\NS^1\rightarrow \R$ such that $a_{11}^k\cdot a_{22}^k-a_{12}^{k}\cdot a_{21}^k>0$. Note that since $f_t$ is a contactomorphism, the condition (iv) is fulfilled by the whole homotopy $\tilde{e}^k_t$. For $t=1$ one can compute 
  \[ 
    d_{(\theta,0,0)} \tilde{e}^k_1(\partial_\theta)=\partial_\theta,
  \] 
  \[
    d_{(\theta,0,0)} \tilde{e}^k_1 (\partial_y)=a_{11}^{k}(\theta)\partial_y+a_{21}^{k}(\theta)\partial_z \text{ and } 
  \] 
  \[ 
    d_{(\theta,0,0)} \tilde{e}^k_1 (\partial_z)=a_{12}^{k}(\theta)\partial_y+a_{22}^{k}(\theta)\partial_z 
  \] 
  and the condition (iv) then implies that $a_{21}^{k}(\theta)=0$, $a_{11}^{k}(\theta)>0$ and $a_{22}^{k}(\theta)>0$. Here, we are using the assumption about the orientation of $\xi$. In particular, 
  \[
    \tilde{e}^k_1(\theta,y,z)=(\theta, a_{11}^{k}(\theta)y+a_{12}^{k}(\theta)z,a_{22}^{k}(\theta)z).
  \] 

  Finally, since $a_{11}^{k}$ and $a_{22}^{k}$ are positive functions we may extend the homotopy of embeddings via 
  \[ 
    \tilde{e}^k_t(\theta,y,z)=(\theta, ((2-t)a_{11}^{k}(\theta)+t-1)y+(2-t)a_{12}^{k}(\theta), ((2-t)a_{22}^{k}(\theta)+t-1)z), (k,t)\in K\times[1,2],
  \] 
  which satisfies the required properties since $\tilde{e}^k_2(\theta,y,z)=(\theta,y,z)$ is a clearly a contactomorphism.
\end{proof}

\section{Legendrian cables}\label{sec:cables}

In this section, we review constructions and properties of Legendrian cables in tight contact $3$-manifolds. Finally, we prove Theorem~\ref{thm:SufficientlyPositiveLinkCables}, \ref{thm:cable} and Corollary \ref{cor:IteratedTorus}.

\subsection{Cabling Legendrian knots}\label{subsec:cables}
Let $K$ be a null-homologous knot in a contact $3$-manifold $(M, \xi)$, $\mu$ a meridian of $K$ and $\lambda$ a Seifert longitude. Take $N$ to be a neighborhood of $K$. For relatively prime integers $p$ and $q$, We define a $(p,q)$-cable $K_{p,q}$ of $K$ to be a knot on $\bd N$ with homology class $p [\lambda] + q [\mu]$ and define the slope of $K_{p,q}$ to be $q/p$. We say $K_{p,q}$ is a \dfn{sufficiently positive cable} of $K$ if $q/p > \tbb(K) + 1$.

Let $L$ be a null-homologous Legendrian knot in a contact $3$-manifold $(M,\xi)$ and $N$ a standard neighborhood of $L$. For $q/p > \tb(L)$, we define the \dfn{standard Legendrian cable} $L_{p,q}$ to be a ruling curve on $\bd N$ of slope $q/p$.

\begin{theorem}[\cite{CEM:cables,EtnyreHonda:UTP}]\label{thm:MaxTBCableItsStandard}
  Let $(M,\xi)$ be a tight contact $3$-manifold, $K$ a null-homologous knot type in $M$ and $K_{p,q}$ a sufficiently positive cable of $K$. Then $\tbb(K_{p,q}) = pq - |q - p\tbb(K)| $. Moreover, if $L$ is a Legendrian representative of $K$ with $\tb(L) = \tbb(K)$, then its standard cable $L_{p,q}$ also has the maximal Thurston-Bennequin number, i.e. $\tb(L_{p,q}) = \tbb(K_{p,q})$.
\end{theorem}

In \cite{CEM:cables}, Chakraborty, Etnyre and the second author classified sufficiently positive Legendrian cables in terms of the classification of the underlying knot. In this paper, we only need a part of this result for Legendrian cables with the maximal Thurston-Bennequin number. 

\begin{theorem}[Chakraborty--Etnyre--Min \cite{CEM:cables}] \label{thm:LegCables-pi0}
  Let $K$ be a knot type in $(S^3,\xi_\std)$ and $K_{p,q}$ a sufficiently positive cable of $K$. Then the space of Legendrian embeddings of $K_{p,q}$ and the space of Legendrian embeddings of $K$ with the maximal Thurston-Bennequin number are isomorphic in the level of $\pi_0$. In other words, 
  \[
   \pi_0(\L^{\tbb}(K_{p,q})) \cong \pi_0(\L^{\tbb}(K)).
  \]
\end{theorem}

\subsection{Proof of Theorem \ref{thm:SufficientlyPositiveLinkCables}}\label{subsec:ProofLinkCables}
Let $L=(L^1,\ldots,L^n)$ be a Legendrian link in a tight contact $3$-manifold $(M,\xi)$ with the maximal Thurston--Bennequin invariants. Consider a sufficiently positive $(B,C)$-cable $L^B_C$ of $L$ such that each component has the maximal Thurston--Bennequin number. Here, $B=\{i_1,\ldots,i_n\}\in\{0,1\}^n$ and $C = \{ (p_1,q_1), \cdots, (p_n,q_n) \}$ where $(p_i,q_i) = (0,0)$, or $q_i/p_i > \tb(L^i) + 1$. Let $(S_j, N_j)$ be a pair of standard neighborhoods of $L^j$ such that $S_j \subset N_j$. We will define $T_j=N_j\setminus (i_j \cdot S_j)$, where $1\cdot S_j=S_j$ and $0 \cdot S_j=\emptyset$. Observe that if $i_j=1$, then $(T_j=\T^2\times[0,1],\xi_j)$ is a Legendrian circle bundle over an annulus and if $i_j=0$ then $T_j = N_j$. 

If $(p_j,q_j) \neq (0,0)$, we realize the cable component $L^j_{p_j,q_j}$ on $\bd N_j$. The complement $\Op(L^j_{p_j,q_j})\cap \bd N_j$ is given by $m_j=\operatorname{gcd}(p_j,q_j)$ of $n_j$-standard annuli for $$n_j=\frac{|p_j\tb(L^j)-q_j|}{m_j}> 1$$ (notice that $n_j > 1$ since $q_j/p_j > \tb(L^j) + 1$).

The complement $(C(L^B_C),\xi)$ is obtained by gluing every $(T_j,\xi)$ to $(C(L),\xi)$ along these $n_j$-standard annuli (if $(p_j,q_j) = (0,0)$, we need not do anything). Therefore, Theorem \ref{thm:GluingAnnuli} applies and there is a homotopy equivalence 
\[
  \mathcal{C}(C(L^B_C),\xi)\cong \mathcal{C}(C(L),\xi)\times \mathcal{C}(T_1,\xi_1)\times\cdots\times\mathcal{C}(T_n,\xi_n).
\] 
Also, $\mathcal{C}(T_j,\xi_j)$ is contractible by Theorem \ref{thm:StandardTightHandlebody} if $i_j=0$ and by \cite[Theorem~1.1.4]{FMP:unknot} if $i_j=1$. This completes the proof. \qed

\subsection{Proof of Theorem \ref{thm:cable}}

The $\pi_0$ case follows from Theorem \ref{thm:LegCables-pi0}. Therefore, in view of Lemma \ref{lem:FramingKnots} and \ref{lem:FramingLegendrians} it is enough to show the existence of a homotopy equivalence 
\[ \Phi: \Cont(C(L),\xi_\std)\times\Z\rightarrow \Cont(C(L_{p,q}),\xi_\std) \] 
to conclude the proof of Theorem \ref{thm:cable}. Here, $L\subseteq (\NS^3,\xi_\std)$ is the image of a Legendrian embedding in $\L^{\tbb}_{(p_0,v_0)}(K)$ and $L_{p,q}$ is a standard Legendrian cable of $L$, which maximizes the $\tb$ number (Theorem \ref{thm:MaxTBCableItsStandard}). Notice that $C(L_{p,q}) = C(L) \cup C_{p,q}$ where $C_{p,q}$ is a cable space, a Seifert fibration over an annulus with one singular fiber. We claim that for sufficiently positive cables, $(C_{p,q},\xi_\std)$ is a Legendrian Seifert fibration (see Definition~\ref{def:lsfs}). To see this, notice that the complement of a Legendrian positive torus knot with the maximal Thurston--Bennequin invariant in $(\NS^3,\xi_\std)$ is a Legendrian Seifert fibration over a disk with two singular fibers, see Proposition~\ref{prop:complement-torus}. If we remove a neighborhood of one singular fiber, we obtain a Legendrian Seifert fibration over an annulus with one singular fiber. Clearly this corresponds to the cable space of a sufficiently positive cable of the Legendrian unknot $U$ with the maximal Thurston--Bennequin invariant in $(\NS^3,\xi_\std)$, which is contactomorphic to the complement of a standard neighborhood of a ruling curve in the interior of a standard neighborhood of $U$. Moreover, this is contactomorphic to the cable space of a sufficiently positive cable of any knot type and this concludes the claim. Thus by Theorem~\ref{thm:lsfs-cmcg}, we have 
\[
  \pi_0(\Cont(C_{p,q},\xi_\std)) = \pi_0(\Diff(\Sigma^0_{2,1})) = \Z \times \Z 
\] 
which is generated by the lifts of Dehn twists along two circles parallel to $\bd \Sigma^0_{2,1}$. If we do not fix one boundary component of $C_{p,q}$, then it will be an infinite cyclic group. We define the required map as 
\begin{align*}
  \Phi:\Cont(C(L),\xi_\std)) \times \Z &\to \Cont(C(L_{p,q}),\xi_\std)\\
  (f,n) &\mapsto f \cup_{T} \tau^n
\end{align*}
where $\tau$ is a contact torus twist along a torus parallel to $\bd C(L_{p,q})$ and $T = \bd C(L)$. It remains to check that $\Phi$ is a homotopy equivalence. 

First we prove that $\Phi$ induces an isomorphism on $\pi_0$. To see the surjectivity, consider $\phi \in \Cont(C(L_{p,q}),\xi_\std)$. We will show that $\phi$ is contact isotopic to $\Phi(f,n)=f\cup_T \tau^n$ for some $f \in \Cont(C(L),\xi_\std)$ and $n \in \Z$. Notice that there exists a smooth isotopy $\phi_t\in\Diff(C(K))$, $t\in[0,1]$, so that $\phi_1=\phi$ and $(\phi_0)_{|\Op(T)}=\Id_{|\Op(T)}$ (\em c.f.~\cite{Budney:knotsS3}\em). Parametrize the torus $T$ by $i:T^2\hookrightarrow C(L_{p,q})$ in such a way that $i_{\NS^1\times\{\theta\}}$ is a Legendrian fiber of $(C_{p,q},\xi_\std)$ for every $\theta\in\NS^1$, so that $i$ is a standard embedding. Since $\phi$ is a contactomorphism it follows that $\phi\circ i$ is also standard and, moreover, smoothly isotopic to $i=\phi_0 \circ i$ in $C(L_{p,q})$ via $i_t=\phi_t\circ i$, $t\in[0,1]$. We should prove that this isotopy could be upgraded to an isotopy through standard embeddings. Since $C_{p,q}$ is a Legendrian Seifert fibration, we can choose Legendrian fibers $F_1=i(\NS^1\times\{\theta_0\})$ on $T$ and $F_2$ on $\bd C(L_{p,q})$. There is an $n$-standard convex annulus $A$ conformed by Legendrian fibers such that $\bd A = F_1 \cup F_2$, so $F_1$ and $F_2$ are Legendrian isotopic. Now consider the smooth isotopy of annuli $\phi_t(A)$, $t\in[0,1]$, note that $\phi_1(A)$ is standard because $\phi=\phi_1$ is a contactomorphism. Since $\bd \phi_t(A) = \phi_t(F_1) \cup \phi_t(F_2)=\phi_t(F_1)\cup F_2$, we conclude that $\phi_1(F_1)=\phi(F_1)$, $F_2$ and $F_1$ are Legendrian isotopic. Moreover, by using the ``ruling'' curves of the annuli $\phi_t(A)$, we can isotope the family of fibers $\phi_t(F_1)$ to the Legendrian fiber $F_1$ in such a way that for $\phi_1(F_1)$ the isotopy is through Legendrian curves. Therefore, we can assume, after a possible contact isotopy, that $\phi\circ i$ and $i$ coincide near $\NS^1\times\{\theta_0\}$, i.e. the fiber $F_1$ is fixed by $\phi$; and, moreover, are smoothly isotopic relative to $\NS^1\times\{\theta_0\}$. Since $L_{p,q}$ is a sufficiently positive cable, we know $n>1$ so we can apply Theorem~\ref{thm:StandardAnnuli} and conclude $\phi\circ i$ and $i$ are contact isotopic. Thus we can assume $\phi$ fixes $T = \bd C(L)$ point-wise so $\phi$ decomposes into $f \cup_T \tau^n$ for some $f \in \Cont(C(L),\xi_\std)$ and $n \in \Z$.  We leave the injectivity part as an exercise for readers since the argument is essentially the same. 

To see that $\Phi$ induces an isomorphism on higher homotopy groups consider the following diagram: 
\begin{displaymath} 
  \xymatrix@M=10pt{
    \Cont_0(C(L),\xi_\std)  \ar@{^{(}->}[d]\ar[r]  & \Diff_0(C(L)) \ar[r] \ar@{^{(}->}[d] & \mathcal{C}(C(L),\xi_\std)   \ar@{^{(}->}[d] \\
    \Cont_0(C(L_{p,q}),\xi_\std) \ar[r] & \Diff_0(C(L_{p,q})) \ar[r] &  \mathcal{C}(C(L_{p,q}),\xi_\std) }
\end{displaymath}
The rows are from Gray fibration and the vertical arrows are induced by the inclusion $i\colon C(L) \hookrightarrow C(L_{p,q}) = C(L) \cup C_{p,q}$. The last two vertical arrows are homotopy equivalences by Budney \cite[Theorem~2.3]{Budney:knotsS3} and Theorem~\ref{thm:SufficientlyPositiveLinkCables}, respectively. Thus by the Five Lemma, the first arrow is also a homotopy equivalence. This completes the proof. \qed

\subsection{Proof of Corollary \ref{cor:IteratedTorus}}
We start proving (ii). As explained in \cite{Budney:knotsS3} the space of long smooth $n$-iterated torus knots has the homotopy type of $(\NS^1)^n$. This follows by applying $n$ times the smooth version of the recursive formula in Theorem \ref{thm:cable} to the space of long smooth unknots, which is contractible by Hatcher \cite{Hatcher:Smale}. Every sufficiently positive $n$-iterated Legendrian torus knot with the maximal $\tb$ number is obtained from the Legendrian unknot with $\tb=-1$ by an application of Theorem \ref{thm:LegCables-pi0} combined with Eliashberg-Fraser result \cite{EliashbergFraser:unknot}. Therefore, since the space of long Legendrian unknots with $\tb=-1$ is homotopy equivalent to its smooth counterpart by \cite{FMP:unknot}, Theorem \ref{thm:cable} readily implies that the inclusion 
\[ 
  i:\L^{\tbb}_{(p_0,v_0)}(K)\hookrightarrow \K_{(p_0,v_0)}(K) 
\]
is a homotopy equivalence for every sufficiently positive $n$-iterated torus knot $K$. The homotopy equivalence stated in (ii) now follows from (\ref{eq:LongHomotopyEquivalence}). 

To prove (i) observe that the homotopy fibers of the inclusion $i$ above and the inclusion 
\[ 
  j:\L^{\tbb}_{(p_0,v_0)}(K,(M,\xi))\hookrightarrow\K_{(p_0,v_0)}(K,M)
\] 
are homotopy equivalent by an application of Proposition \ref{prop:FromS3ToM3} (see the Remark below the Proposition). In particular, $\operatorname{Hofiber}(j)$ is contractible. This, combined with Theorem \ref{thm:LegCables-pi0} and Eliashberg-Fraser result again, implies the result. \qed

\section{Legendrian Seifert fibered spaces}\label{sec:lsfs}

In this section, we define a \dfn{Legendrian Seifert fibered space} and determine the homotopy type of the group of contactomorphisms of any Legendrian Seifert fibration over a surface with boundary.

\subsection{Definitions and properties}\label{subsec:lsfs-def}
First, we review the basics of Seifert fibered spaces briefly. Let $\Sigma^g_n$ be a genus $g$ surface with $n$ boundary components and $\Sigma^g_{n,m}$ a genus $g$ surface with $n$ boundary components and $m$ punctures. Consider $\Sigma^g_n \times \NS^1$. For a fiber $F = \{p\} \times \NS^1$, we define the \dfn{vertical framing} of $F$ to be the product framing of $\Sigma^g_n \times \NS^1$. Then we define a Seifert fibered space $M(\Sigma^g_n,r_1,\dots,r_m)$ to be the result of Dehn $r_i$-surgery on $m$ fibers of $\Sigma^g_n \times \NS^1$ with respect to the vertical framing. Each surgery dual knot is called a \dfn{singular fiber} of the Seifert fibration. Notice that the complement of all singular fibers is a circle bundle over $\Sigma^g_{n,m}$ and each fiber of this circle bundle is called a \dfn{regular fiber} of the Seifert fibration. 

\begin{definition}\label{def:lsfs}
  A contact manifold $(M,\xi)$ is a \dfn{Legendrian Seifert fibered space} if $M$ is a Seifert fibration and all regular and singular fibers are Legendrian. Also, let $t$ be the twisting number of a Legendrian regular fiber with respect to the vertical framing. We call this $t$ the \dfn{twisting number} of the Legendrian Seifert fibered space. 
\end{definition}

Notice that if $(M,\xi)$ is a Legendrian Seifert fibration over $\Sigma^g_n$ with $m$ singular fibers, then the complement of the singular fibers is a Legendrian circle bundle over $\Sigma^g_{n,m}$. The next proposition verifies that the twisting number of a Legendrian Seifert fibration is well-defined.

\begin{proposition}\label{prop:lsfs-tw}
  Let $(M,\xi)$ be a Legendrian Seifert fibration. Then a Legendrian regular fiber has the maximal twisting number in the smooth isotopy class of the regular fibers. 
\end{proposition}

\begin{proof}
  Suppose there is a Legendrian knot $L$ in $(M,\xi)$ which is smoothly isotopic to a regular fiber with a larger twisting number. After a small perturbation, we can assume that $L$ is disjoint from all singular fibers by transversality. Thus we may assume that $L$ is in the Legendrian circle bundle $(M',\xi')$, the complement of all singular fibers. By \cite[Lemma~3.6]{Giroux:bundles}, Legendrian fibers in a Legendrian circle bundle have the maximal twisting number in the smooth isotopy class, so it is a contradiction. 
\end{proof}

We observe that the contact structure in a neighborhood of a singular fiber is completely determined.

\begin{lemma}\label{lem:lsfs-one-singular}
  A Legendrian Seifert fibration over a disk with one singular fiber is contactomorphic to a standard contact solid torus, i.e. a tight convex solid torus with two longitudinal dividing curves.
\end{lemma}

\begin{proof}
  Let $(M,\xi)$ be a Legendrian Seifert fibration over a disk with one singular fiber. Since $M$ is diffeomorphic to a solid torus, $(M,\xi)$ is a contact solid torus with convex boundary. We first show that $(M,\xi)$ is tight. Recall that in an overtwisted contact structure, any smooth knot type admits a Legendrian representative with any contact framing. By Proposition~\ref{prop:lsfs-tw}, however, the regular fibers have the maximal twisting number, so $\xi$ must be tight. Next we claim that the boundary of $(M,\xi)$ has two dividing curves. Assume $(M,\xi)$ has more than two dividing curves. Then inside of $M$, we can find a neighborhood of a singular fiber with the same dividing slope but two dividing curves. Then the ruling curves on this neighborhood have a larger twisting number than the regular fibers, which contradicts Proposition~\ref{prop:lsfs-tw}. This completes the claim. 
  
  Now assume that the dividing curves are not longitudes of $M$. Take an $n$-fold covering of $(M,\xi)$ and denote it by $(\widetilde{M},\widetilde{\xi})$. Since the covering operation preserves the fibers, $(\widetilde{M},\widetilde{\xi})$ is also a Legendrian Seifert fibration.  Since the dividing curves are not longitudinal, for sufficiently large $n$ there are more than two dividing curves on $(\widetilde{M},\widetilde{\xi})$. This contradicts the claim above so the dividing curves must be longitudinal.  
\end{proof}

We finish this section by showing that any diffeomorphism of the base (after a perturbation, if necessary) can be lifted to a contactomorphism of a Legendrian Seifert fibration.

\begin{proposition}\label{prop:lift}
  Let $(M,\xi)$ be a Legendrian Seifert fibration over $\Sigma^g_n$ with $m$ singular fibers. For any component of $\Diff(\Sigma^g_{n,m})$, there exists a diffeomorphism $f$ that can be lifted to a contactomorphism $\widetilde{f} \in \Cont(M,\xi)$.
\end{proposition}

\begin{proof}
  First, remove a small neighborhood of each singular fiber and obtain a Legendrian circle bundle $(M',\xi')$ over $\Sigma^g_{n+m}$. Let $f$ be a diffeomorphism of $\Sigma^g_{n,m}$. After isotopy, we can assume $f$ is the identity near the punctures. Now we can lift $f|_{\Sigma^g_{n+m}}$ to a contactomorphism $\widetilde{f}$ of the Legendrian circle bundle $(M',\xi')$ (\dfn{c.f.} \cite[Corollary~2.8]{GirouxMassot:bundles}). Now extend $\widetilde{f}$ by the identity over the neighborhoods of the singular fibers and we obtain a contactomorphism $\widetilde{f} \in \Cont(M,\xi)$ which is a lift of $f$.
\end{proof}

\subsection{Proof of Theorem \ref{thm:Main-LSB}} 
Here, we will prove Theorem \ref{thm:Main-LSB}. First, we will consider the part (i) of the statement. 

\begin{theorem}\label{thm:lsfs-higher-homotopy}
  Let $(M,\xi)$ be a Legendrian Seifert fibered space over $\Sigma^g_n$ with $m$ singular fibers and $n\geq 1$. If the twisting number of regular fibers is $-1$, then we assume the regular fiber is nontrivial in $\pi_1(M)$.  Then, the space $\mathcal{C}(M,\xi)$ is contractible. 
\end{theorem}

\begin{proof}
  We will apply Theorem \ref{thm:GluingAnnuli} several times to prove the result. Note that, when the twisting number of a regular fiber is $t=-1$ the result would still apply because we can unwrap the regular fibers by passing to a covering of $M$ and they will still maximize the twisting number (see Theorem \ref{thm:StandardAnnuli}). 
    
  We proceed by induction over $k = g + n$. Assume that the statement is true for $k$. Now we will prove it for $k+1$. Let $\alpha$ be a properly embedded essential arc on $\Sigma^g_{n,m}$ such that $\Sigma^g_{n,m} \setminus \alpha = \Sigma^{g'}_{n',m}$ where $(g',n') = (g-1,n)$ or $(g,n-1)$. Now let $A := \alpha \times \NS^1$, an essential annulus in $M$ and consider the inclusion $j:A \hookrightarrow M$. Since $A$ is fibered by regular fibers, we may assume that it is a $|t|$-standard convex annulus by Lemma~\ref{prop:lsfs-tw}. Now consider the complement of an $I$-invariant neighborhood of $A$
  \[
    (M',\xi') := (M\setminus N(A),\xi).
  \] 
  Here, $(M',\xi')$ is a Legendrian Seifert fibration over $\Sigma^{g'}_{n'}$ with $m$ singular fibers where $g'+n'=k$. In particular, $\mathcal{C}(M',\xi')$ is contractible by hypothesis. Since $(M,\xi)$ is obtained from $(M',\xi')$ by identifying both boundaries of $N(A)$ we may apply Theorem \ref{thm:GluingAnnuli} to conclude that $\mathcal{C}(M,\xi)$ is contractible. 

  It remains to check the statement for $k=1$. In this case, by the assumption in the statement of the theorem, we have $g=0$ and $n=1$. Let $\alpha_1,\ldots,\alpha_{m-1}$ be properly embedded disjoint essential arcs on $\Sigma^0_{1,m}$ such that each component of $\Sigma^0_{1,m} \setminus \bigcup_{i=1}^{m-1} \alpha_i$ contains a single puncture. Let $A_i := \alpha_i \times \NS^1$. Then the complement of $I$-invariant neighborhoods of $A_i$ is  
  \[
    (M \setminus \bigcup_{i=1}^{m-1} N(A_i), \xi) = \bigcup_{i=1}^m (V_i,\xi_i)
  \] 
  where each $(V_i,\xi_i)$ is a Legendrian Seifert fibration over a disk with one singular fiber. It follows again from Theorem \ref{thm:GluingAnnuli} that 
  \[ 
    \mathcal{C}(M,\xi)\cong \mathcal{C}(V_1,\xi_1)\times\cdots\times\mathcal{C}(V_{m-1},\xi_{m-1}).
  \]
  By Lemma~\ref{lem:lsfs-one-singular}, each $(V_i,\xi_i)$ is a standard contact solid torus. Therefore, the result follows from Theorem~\ref{thm:StandardTightHandlebody}.
\end{proof}

It follows from the Gray fibration  
\[
  \Cont_0(M,\xi)\rightarrow \Diff_0(M)\rightarrow \mathcal{C}(M,\xi)
\]
that the inclusion $\Cont_0(M,\xi)\hookrightarrow \Diff_0(M)$ is a homotopy equivalence for every Legendrian Seifert fibration with boundary. It follows then that $\Cont(M,\xi)$ is homotopically discrete since $\Diff_0(M)$ is contractible (for any Seifert fibration $M$ with boundary) by a result of A. Hatcher \cite{Hatcher:fibering}. To prove part (ii) of Theorem \ref{thm:Main-LSB}, it remains to determine the contact mapping class group of $(M,\xi)$.

\begin{theorem}\label{thm:lsfs-cmcg}
  Let $(M,\xi)$ be a Legendrian Seifert fibration over $\Sigma^g_n$ with $m$ singular fibers and $n\geq 1$. If the twisting number of regular fibers is $-1$, then we assume the regular fiber is nontrivial in $\pi_1(M)$. Then  
  \[
    \pi_0(\Cont(M,\xi)) \cong \pi_0(\Diff(\Sigma^g_{n,m})).
  \] 
\end{theorem}

\begin{proof}
  By Proposition~\ref{prop:lift}, we can lift a diffeomorphism of $\Sigma^g_{n,m}$ to a contactomorphism of $(M,\xi)$ and it induces a well-defined map $\pi_0(\Diff(\Sigma^g_{n,m})) \to \pi_0(\Cont(M,\xi))$ in the $\pi_0$-level. We will show this map is an isomorphism.  

  We first deal with injectivity. Notice that the induced map
  $i_*\colon \pi_1(\Sigma^g_{n,m}) \to \pi_1(M)$ from the inclusion is injective. This implies that if $f \in \Diff(\Sigma^g_{n,m})$ acts nontrivially on $\pi_1(\Sigma^g_{n,m})$ then its lift $\widetilde{f}$ also acts nontrivially on $\pi_1(M)$. Thus if $f$ is not smoothly isotopic to the identity, then its lift also cannot be smoothly isotopic to the identity, which implies injectivity.  

  Now we consider surjectivity. Let $\phi \in \Cont(M,\xi)$ be a contactomorphism of $(M,\xi)$. By the work of Waldhausen \cite{Waldhausen:SFS} (see also Hatcher \cite{Hatcher:fibering}), $\phi$ is smoothly isotopic to a fiber preserving diffeomorphism over some $f \in \Diff(\Sigma^g_{n,m})$. Let $\widetilde{f} \in \Cont(M,\xi)$ be the lift of $f$ by Proposition~\ref{prop:lift}. We claim that $\psi := \phi^{-1} \circ \widetilde{f}$ is contact isotopic to the identity. This will imply that $\phi$ is contact isotopic to $\widetilde{f}$ and will complete the proof. Let $k: = 2g+n+m-2$ and $\alpha_1, \ldots, \alpha_k$ be properly embedded essential arcs on $\Sigma^g_{n,m}$ such that 
  \[
    \Sigma^g_{n,m} \setminus \bigcup_{i=1}^k \alpha_i = \bigcup_{i=1}^m \Sigma^0_{1,1},
  \]
  a disjoint union of $m$ disks with a single puncture. Let $A_i := \alpha_i \times \NS^1$ be properly embedded essential annuli in $(M,\xi)$. Since they are fibered by Legendrian regular fibers, we may assume that they are $|t|$-standard convex annuli by Proposition~\ref{prop:lsfs-tw}. Also, $\psi$ is smoothly isotopic to a fiber preserving diffeomorphism over the identity of $\Sigma^g_{n,m}$ and this implies that $\psi(A_i)$ is smoothly isotopic to $A_i$ setwise. Moreover, since $t \leq -2$, $\psi(A_i)$ is contact isotopic to $A_i$ setwise by Theorem~\ref{thm:StandardAnnuli}. Thus we may assume that $\psi$ fixes $A_i$ setwise. By Proposition~\ref{prop:rigidityContacto}, we may assume that $\psi|_{N(A_i)}$ is the identity where $N(A_i)$ is an $I$-invariant neighborhood of $A_i$. Also, the complement of the union of $N(A_i)$ in $(M,\xi)$ is 
  \[
    (M,\xi) \setminus \bigcup_{i=1}^k N(A_i) = \bigcup_{i=1}^m (V_i,\xi_i),
  \]  
  where each $(V_i,\xi_i)$ is a Legendrian Seifert fibration over a disk with one singular fiber. By Lemma~\ref{lem:lsfs-one-singular}, each $(V_i,\xi_i)$ is a standard contact solid torus and by Theorem~\ref{thm:StandardTightHandlebody}, $\psi|_{V_i}$ is contact isotopic to the identity. Therefore, $\psi$ is contact isotopic to the identity and this completes the claim. 
\end{proof}

For later use, we explicitly determine the group of contactomorphisms of a Legendrian Seifert fibration over a disk.

\begin{corollary}\label{cor:lsfs-pure-braid}
  Let $(M,\xi)$ be a Legendrian Seifert fibration over $\NS^2_{n,m}$ with $n\geq1$. Then
   $$ \Cont(M,\xi) \cong \PB_{n+m-1} \times \Z^{n-1} $$
\end{corollary}

\begin{proof}
  In view  of Theorem \ref{thm:Main-LSB} it is enough to check that $\Diff(\NS^2_{n,m})$ is homotopy equivalent to $\PB_{n+m}\times \Z^{n-1}$. Denote by $C_{n+m-1}(\D^2)$ the ordered configuration space of $n+m-1$ points in $\D^2$, and by the projection over the first $n-1$ points $\pi:C_{n+m-1}(\D^2)\rightarrow (\D^2)^{n-1}, (b_1,\ldots,b_{n+m-1})\mapsto (b_1,\ldots,b_{n-1})$. Consider the total space of the pull-back bundle $\pi^*\Fr^+(\D^2)^{n-1}=C_{n+m-1}(\D^2)\times \GL^+(2,\R)^{n-1}$, where $\Fr^+(\D^2)^{n-1}$ is the product of $n-1$ oriented frame bundles over $\D^2$.

  Consider the fibration $$ \Diff(\NS^2_{n,m})=\Diff(\D^2_{n-1,m})\hookrightarrow \Diff(\D^2)\rightarrow C_{n+m-1}(\D^2)\times \GL^+(2,\R)^{n-1}, $$ induced by evaluation. The group $\Diff(\D^2)$ is contractible by a result of Smale \cite{Smale:DiffeoDisk}. The result follows from observing that $C_{n+m-1}(\D^2)\times \GL^+(2,\R)^{n-1}$ is a $K(G,1)$ space with 
  $$ G=\pi_1(C_{n+m-1}(\D^2))\times\pi_1(\GL^+(2,\R)^{n-1})\cong \PB_{n+m-1}\times \Z^{n-1}. $$
\end{proof}

\section{Legendrian Seifert fibered links}\label{sec:lsflink}

A \em Seifert fibered link \em is a link in $\NS^3$ of which the complement is a Seifert fibered space. Such links were completely classified by Burde and Murasugi \cite{BurdeMurasugi:SFlink} (see also \cite{Budney:SFlink}). Embedding spaces of Seifert fibered links have been recently studied by Havens and Koytcheff \cite{HavensKoytcheff:knots}. In this section, we will classify the Legendrian Seifert fibered links with the maximal Thurston--Bennequin number up to Legendrian isotopy and prove that the complement of standard neighborhoods of these links is a Legendrian Seifert fibration, except when one of the link components is a negative torus knot. Lastly, we will prove Theorem \ref{thm:main-SFlinkS3}.

\subsection{Torus links}\label{subsec:torus}
Let $(p,q)$ be a pair of relatively prime positive integers. For $n \geq 1$, we denote the $(np,nq)$-torus link by $T_{np,nq}$. The isotopy classes of Legendrian torus knots (when $n = 1$) were classified by Etnyre and Honda \cite{EtnyreHonda:torusKnots}. In particular, the maximal Thurston--Bennequin number of $T_{p,q}$ is 
\[
  \tbb(T_{p,q}) = pq - p - q
\]
and there exists a unique Legendrian representative with the maximal Thurston--Bennequin number. See the left drawing of Figure~\ref{fig:torus} for a Legendrian right-handed trefoil with the maximal Thurston--Bennequin number. 

\vspace*{0.3cm}
\begin{figure}[htbp]{\scriptsize
  \begin{overpic}[tics=20]{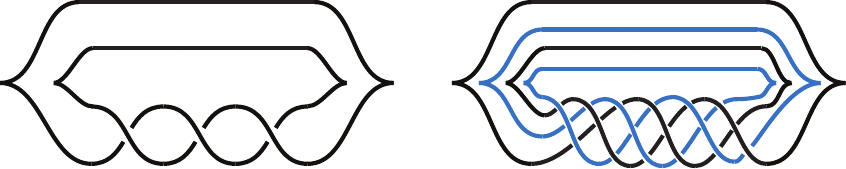}
  \end{overpic}}
  \vspace{0.2cm}
  \caption{Legendrian $(2,3)$-torus knot and $(4,6)$-torus link with the maximal Thurston--Bennequin number.}
  \label{fig:torus}
\end{figure}

The isotopy classes of Legendrian torus links (when $n > 1$) were classified by Dalton, Etnyre and Traynor \cite{DET:torusLinks}. Let $T^i_{p,q}$ be the components of $T_{np,nq}$ for $i = 1,...,n$. Then the maximal Thurston--Bennequin number of $T_{np,nq}$ is 
\[
  \tbb(T_{np,nq}) = \sum_{i=1}^n \tbb(T^i_{p,q}) = n(pq - p - q)
\]
and there exists a unique Legendrian representative with the maximal Thurston--Bennequin number. See the right drawing of Figure~\ref{fig:torus} for a Legendrian $T_{4,6}$ with the maximal Thurston-Bennequin number. The following theorem summarizes the discussions above.

\begin{theorem}[\cite{DET:torusLinks, EtnyreHonda:torusKnots}]\label{thm:pi0-torus}
  For $n \geq 1$, the space of Legendrian embeddings of a positive $(np,nq)$-torus link with the maximal Thurston--Bennequin number is connected. In short, 
  \[
    \pi_0(\L^{\tbb}(T_{np,nq})) = 1.
  \]
\end{theorem}

It is well known ({\it c.f.}~\cite{GeigesOnaran:nonloose, Martelli:book}) that the complement of a $(np,nq)$-torus link is a Seifert fibration $M(\NS^2_n, -p'/p, -q'/q)$ where $\NS^2_n = \NS^2 \setminus \bigcup_n\D^2$ and $p'$, $q'$ are integers satisfying $pq'+p'q = 1$.

Etnyre, LaFountain and Tosun \cite{ELT:torus} showed that in $(\NS^3,\xi_\std)$ the complement of a standard neighborhood of a Legendrian positive $(p,q)$-torus knot with the maximal Thurston--Bennequin number is a Legendrian Seifert fibration with the twisting number $-(p+q)$. Here, we generalize this result to links by building the model of the complement of a Legendrian positive torus link by following \cite[Section~3.1]{ELT:torus}. 

\begin{figure}[htbp]{\scriptsize
  \begin{overpic}[tics=20]{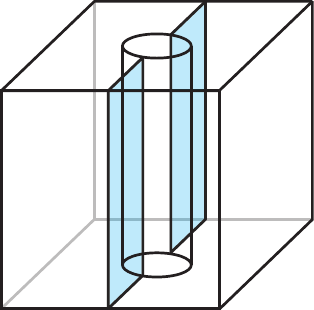}
    \put(35, 88){$A \rightarrow$}
    \put(71, 32){$N$}
  \end{overpic}}
  \vspace{0.2cm}
  \caption{The complement $C$ of a neighborhood of a $(np,nq)$--torus link (here, $n=1$). Each cube is $T^2\times [0,1]$ (the top and bottom are identified as are the front and back) with curves on the right and left collapsed to give $\NS^3$. We see that annulus $A$ separates $C$ into two solid tori.}
  \label{fig:complement}
\end{figure}

Let $H = U_1 \sqcup U_2$ be the Legendrian Hopf link with the maximal Thurston--Bennequin number in $(\NS^3,\xi_\std)$ and $V_1$ and $V_2$ standard neighborhoods of $U_1$ and $U_2$, respectively. Then there is a convex torus $T$ with slope $-1$ in the complement of the Hopf link that separates $\NS^3$ such that $\NS^3 = V_1 \cup (T \times [-1,1]) \cup V_2$ where $T \times [-1,1]$ is an $I$-invariant neighborhood of $T = T \times \{0\}$. Here we use the coordinate system on any torus parallel to $T$ coming from the Seifert framing of $V_1$ (so the meridian of $V_1$ has slope $\infty$ and the meridian of $V_2$ has slope $0$). We may take a Legendrian $(np,nq)$-torus link $L$ that sits on $T$ as ruling curves of slope $p/q$. Let $L_i$ be the components of $L$ for $i = 1,\ldots,n$. Each ruling curve has the maximal Thurston--Bennequin number since $\tb(L_i) = pq - \frac12 |\Gamma \cap L_i|$ and 
\begin{align}\label{eq:intersection}
  |\Gamma \cap L_i|  = 2\left|\frac pq \bigcdot (-1)\right| = 2(p+q)
\end{align}
so $\tb(L_i) = pq - p - q$. Therefore, 
\[
  \tb(L) = n(pq-p-q) = \tbb(T_{np,nq})
\]
so $L$ is a Legendrian torus link with the maximal Thurston--Bennequin number. Let $N = N_1 \sqcup \cdots \sqcup N_n$ be a standard neighborhood of $L$ in $T\times [-1,1]$ and 
\[
  (C(T_{np,nq}), \xi_\std) := (\NS^3\setminus N, \xi_\std|_{\NS^3\setminus N}).
\]
Let $A = A_1 \sqcup \cdots \sqcup A_n := T \cap C(T_{p,q})$, a disjoint union of essential annuli. Then we can consider $T \times [-1,1]$ as the union of $N$ and $N(A)$, a neighborhood of the annuli $A$ in $C(T_{np,nq})$. See Figure~\ref{fig:complement}. We can assume that each boundary component of $A$ is a ruling curve on $T$. Since the ruling curves are Legendrian $(p,q)$-torus knots with the maximal Thurston--Bennequin number, $A$ is $(p+q)$-standard convex annuli by~(\ref{eq:intersection}). Since we can also think $N(A)$ as an $I$-invariant neighborhood of $A$, we have the following model for $C(T_{np,nq})$,
\[
  C(T_{np,nq}) = V_1 \cup N(A) \cup V_2.
\]
Etnyre, LaFountain and Tosun \cite{ELT:torus} also showed that $V_1$ and $V_2$ are fibered by ruling curves except for the cores, which are a Legendrian Hopf link in $(\NS^3,\xi_\std)$. Also, since each $N(A_i)$ is an $I$-invariant neighborhood of a standard convex annulus, it is fibered by vertical ruling curves. The next proposition summarizes the discussions so far.

\begin{proposition}\label{prop:complement-torus}
  Let $(C(T_{np,nq}),\xi_\std)$ be the complement of a standard neighborhood of a Legendrian positive torus link $T_{np,nq}$ in $(\NS^3,\xi_\std)$ with the maximal Thurston--Bennequin number. Then $(C(T_{np,nq}),\xi_\std)$ is a Legendrian Seifert fibration over $\NS^2_n$ with two singular fibers with the twisting number $-(p+q)$.
\end{proposition}

\subsection{Seifert links}\label{subsec:Seifert}
For $n \geq 1$, a \dfn{Seifert link $S_{np,nq}$} is a $(n+1)$-component link that is the union of an unknot $U$, which is a core of a Heegaard torus of $\NS^3$, and a torus link $T_{np,nq}$, which is a $(np,nq)$-cable of $U$ on the Heegaard torus. See the left drawing of Figure~\ref{fig:Seifert} for $S_{2,3}$.

\begin{remark}
  Notice that ``Seifert fibered link'' and ``Seifert link'' are different.
\end{remark}

\begin{figure}[htbp]{\scriptsize
  \begin{overpic}[tics=20]{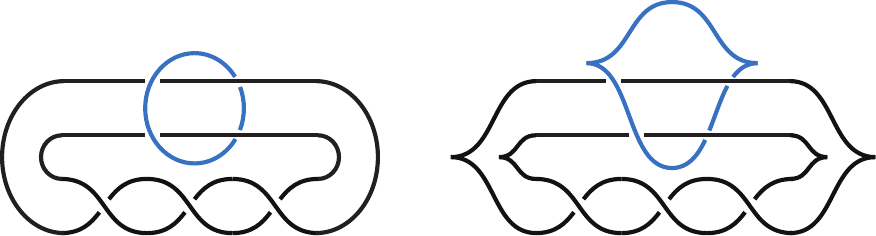}
  \end{overpic}}
  \vspace{0.2cm}
  \caption{Seifert links.}
  \label{fig:Seifert}
\end{figure}

We say $S_{np,nq}$ is \dfn{positive} if $(p,q)$ is a pair of relatively prime positive integers and $n\geq 1$. In this case, the maximal Thurston--Bennequin number of $S_{np,nq}$ is 
\[
  \tbb(S_{np,nq}) = \tbb(U) + \tbb(T_{np,nq}) = n(pq-p-q)-1.
\]
See the right drawing of Figure~\ref{fig:Seifert} for a Legendrian $S_{2,3}$ with the maximal Thurston--Bennequin number. There exists a unique Legendrian representative of $S_{np,nq}$ with the maximal Thurston-Bennequin number up to Legendrian isotopy. For $n=1$, it was shown by Ding and Geiges \cite{DingGeiges:links}.

\begin{theorem}\label{thm:pi0-Seifert}
  For $n \geq 1$, the space of Legendrian embeddings of positive Seifert link $S_{np,nq}$ with the maximal Thurston--Bennequin number is connected. In short,   
  \[
    \pi_0(\L^{\tbb}(S_{np,nq})) = 1.
  \]
\end{theorem}

\begin{proof}
  Let $(U_1,T_1)$ and $(U_2,T_2)$ be two Legendrian representatives of an $(np,nq)$-Seifert link with the maximal Thurston--Bennequin number. Here, $U_i$ is a Legendrian unknot with $\tb = -1$ for $i=1,2$, and $T_i$ is a Legendrian positive $(np,nq)$-torus link with the maximal Thurston--Bennequin number. Let $V_i$ be a neighborhoods of $T_i$. After isotopy, we can assume that $T_i$ lies on $\bd V_i$. Each component of $L_i$ has the maximal Thurston--Bennequin number $pq-p-q$, which is less than the framing induced by the torus $\bd V_i$, which is $pq$. Thus we can perturb $\bd V_i$ to be convex while fixing $T_i$. Let $s_i$ be the dividing slope of $\bd V_i$. By \cite[Lemma~2.2]{EtnyreHonda:UTP}, we have
  \[
    pq - p - q = pq - |s_i \bigcdot p/q|.
  \] 
  This implies that $s_i = -1$ and by Giroux flexibility, we may assume that $\bd V_i$ is a standard convex torus and $T_i$ lies on $\bd V_i$ as ruling curves. Thus $V_i$ is a standard neighborhood of $U_i$. By Eliashberg and Fraser \cite{EliashbergFraser:unknot}, the unknot is Legendrian simple and $U_1$ and $U_2$ are Legendrian isotopic. Therefore, there is a contact isotopy from $V_1$ and $V_2$ and we may assume that both $T_1$ and $T_2$ are ruling curves on $\bd V_1$. After isotopy through ruling curves, we can match the two links. However, as ordered Legendrian links, there could be some permutation among the components of the links. In this case, we can swap any two link components of $T_i$ in an $I$-invariant neighborhood of $\bd V_1$ according to \cite[Proof of Theorem~6.1]{DET:torusLinks}, and this completes the proof. 
\end{proof}

Let $C(S_{np,nq}) := \NS^3 \setminus N(S_{np,nq})$, the complement of a neighborhood of a Seifert link $S_{np,nq}$. In Section~\ref{subsec:torus}, we already noticed that the torus link complement is a Seifert fibration with two singular fibers. Since $S_{np,nq}$ is the union of a torus link $T_{np,nq}$ and a singular fiber in the complement, $C(S_{np,nq}) = M(\NS^2_{n+1}, -p'/p)$ where $\NS^2_{n+1} = \NS^2 \setminus \bigcup_{n+1} \D^2$ and $p'$, $q'$ are integers satisfying $pq' + p'q = 1$. The next proposition characterizes the contact structure on the complement of a standard neighborhood of a Legendrian positive Seifert link with the maximal Thurston--Bennequin number. 

\begin{proposition}\label{prop:complement-Seifert}
  Let $(C(S_{np,nq}),\xi_\std)$ be the complement of a standard neighborhood of a Legendrian positive Seifert link $S_{np,nq}$ in $(\NS^3,\xi_\std)$ with the maximal Thurston--Bennequin number. Then $(C(S_{np,nq}),\xi_\std)$ is a Legendrian Seifert fibration over $\NS^2_{n+1}$ with one singular fiber and the twisting number $-(p+q)$.  
\end{proposition}

\begin{proof}
  In Proposition~\ref{prop:complement-torus}, we already showed that the complement of a Legendrian torus link with the maximal Thurston--Bennequin number is a Legendrian Seifert fibration with the twisting number $-(p+q)$. Also, $C(S_{np,nq})$ is obtained by removing a neighborhood of a singular fiber and this completes the proof.  
\end{proof}

\subsection{Seifert-Hopf links}\label{subsec:Seifert-Hopf}
For $n \geq 1$, a \dfn{Seifert-Hopf link $R_{np,nq}$} is an $(n+2)$-component link that is the union of a Hopf link $U_1 \sqcup U_2$, which is the cores of solid tori bounded by a Heegaard torus of $\NS^3$, and a torus link $T_{np,nq}$, which is a $(np,nq)$-cable of $U_1$ on the Heegaard torus. See the left drawing of Figure~\ref{fig:Seifert-Hopf} for $R_{2,3}$.

We say $R_{np,nq}$ is \dfn{positive} if $(p,q)$ is a pair of relatively prime positive integers and $n\geq 1$. In this case, the maximal Thurston--Bennequin number of $R_{np,nq}$ is 
\[
  \tbb(R_{np,nq}) = \tbb(U_1) + \tbb(U_2) + \tbb(T_{np,nq}) = n(pq-p-q)-2.
\]
See the right drawing of Figure~\ref{fig:Seifert-Hopf} for Legendrian $R_{2,3}$ with the maximal Thurston-Bennequin number. There exists a unique Legendrian representative of positive Seifert Hopf link $R_{np,nq}$ with the maximal Thurston-Bennequin number up to Legendrian isotopy.

\begin{figure}[htbp]{\scriptsize
  \begin{overpic}[tics=20]{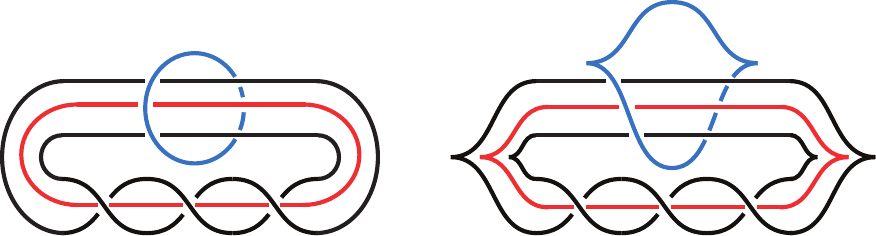}
  \end{overpic}}
  \vspace{0.2cm}
  \caption{Seifert-Hopf links.}
  \label{fig:Seifert-Hopf}
\end{figure}

\begin{theorem}\label{thm:pi0-Sefert-Hopf}
  For $n \geq 1$, the space of Legendrian embeddings of a positive Seifert-Hopf link $R_{np,nq}$ with the maximal Thurston--Bennequin number is connected. In short,
  \[
    \pi_0(\L^{\tbb}(R_{np,nq})) = 1.    
  \]
\end{theorem}

\begin{proof}
  The proof is almost the same as the one given in Theorem \ref{thm:pi0-Seifert}. We leave it as an exercise to readers.
\end{proof}

Let $C(R_{np,nq}) := \NS^3 \setminus N(R_{np,nq})$, the complement of a neighborhood of a Seifert-Hopf link $R_{np,nq}$. It is the complement of two singular fibers in $C(T_{np,nq})$, so $C(R_{np,nq}) = \NS^2_{n+2} \times \NS^1$ where $\NS^2_{n+2} = \NS^2 \setminus \bigsqcup_{n+2} \D^2$. The next proposition characterizes the contact structure on the complement of a standard neighborhood of a Legendrian positive Seifert-Hopf link with the maximal Thurston--Bennequin number. 

\begin{proposition}\label{prop:complement-Seifert-Hopf}
  Let $(C(R_{np,nq}),\xi_\std)$ be the complement of a standard neighborhood of a Legendrian positive Seifert-Hopf link $R_{np,nq}$ in $(\NS^3,\xi_\std)$ with the maximal Thurston--Bennequin number. Then $(C(R_{np,nq}),\xi_\std)$ is a Legendrian circle bundle over $\NS^2_{n+2}$ with the twisting number $-(p+q)$.
\end{proposition}

\begin{proof} 
  The proof is almost the same as the one given in Proposition~\ref{prop:complement-Seifert}. We leave it as an exercise to readers.
\end{proof}

\subsection{Keychain links}\label{subsec:keychain}
For $n\geq2$, a \dfn{keychain link $H_n$} is a connected sum of $(n-1)$ copies of Hopf links, which is a $n$-component link as shown in the left drawing of Figure~\ref{fig:keychain}.

Let $U_i$ be the components of $H_n$ for $i = 1,...,n$. Then the maximal Thurston--Bennequin number of $H_n$ is 
\[
  \tbb(H_n) = \sum_{i=1}^{n} \tbb(U_i) = -n. 
\]
See the right drawing of Figure~\ref{fig:keychain} for Legendrian $H_n$ with the maximal Thurston--Bennequin number. There exists a unique Legendrian representative of a keychain link with maximal Thurston-Bennequin number up to Legendrian isotopy.

\begin{figure}[htbp]{\scriptsize
  \begin{overpic}[tics=20]{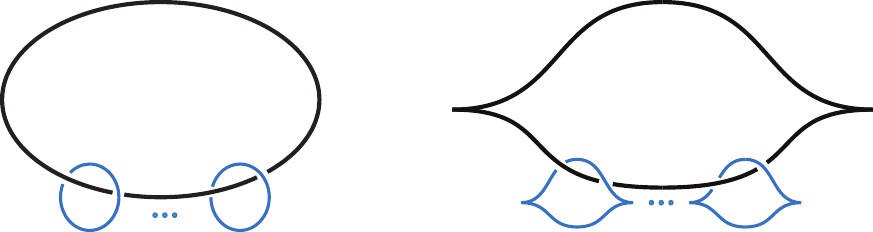}
  \end{overpic}}
  \vspace{0.2cm}
  \caption{Keychain links}
  \label{fig:keychain}
\end{figure}

\begin{theorem}\label{thm:pi0-keychain}
  The space of Legendrian embedding of a keychain link $H_n$ with the maximal Thurston--Bennequin number is connected. In short, 
  \[
    \pi_0(\L^{\tbb}(H_n)) = 1.
  \]   
\end{theorem}

\begin{proof}
  The proof is almost the same as the one given in Theorem \ref{thm:pi0-Seifert}. We leave it as an exercise to readers. 
\end{proof}

Let $C(H_n) := \NS^3 \setminus N(H_n)$, the complement of a neighborhood of a keychain link $H_n$. It is straightforward to see that $C(H_n) = \NS^2_{n} \times \NS^1$ where $\NS^2_{n} = \NS^2 \setminus \bigsqcup_n \D^2$. The next proposition characterizes the contact structure on the complement of a standard neighborhood of a Legendrian positive Seifert-Hopf link with the maximal Thurston--Bennequin number. 

\begin{proposition}\label{prop:complement-keychain}
  Let $(C(H_n),\xi_\std)$ be the complement of a standard neighborhood of a Legendrian keychain link $H_n$ in $(\NS^3,\xi_\std)$ with the maximal Thurston--Bennequin number. Then $(C(H_n),\xi_\std)$ is a Legendrian circle bundle over $\NS^2_n$ with the twisting number $-1$.
\end{proposition}

\begin{proof}
  The proof is almost the same as the one given in Proposition~\ref{prop:complement-Seifert}. We leave it as an exercise to readers.
\end{proof}

\subsection{Proof of Theorem \ref{thm:main-SFlinkS3}}
Let $L\subseteq (\NS^3,\xi_\std)$ be a Legendrian positive Seifert fibered link with the maximal Thurston-Bennequin invariant. It follows from Propositions \ref{prop:complement-torus}, \ref{prop:complement-Seifert}, \ref{prop:complement-Seifert-Hopf} and \ref{prop:complement-keychain} that $(C(L),\xi_\std)$ is a Legendrian Seifert fibered space. It follows from Theorem \ref{thm:Main-LSB} that $L$ satisfies the $C$-property. By Lemma \ref{lem:LegVSCont}, Corollary \ref{cor:lsfs-pure-braid} and the homotopy equivalence (\ref{eq:HomotopyEquivalenceLongLinks}), we can determine the higher homotopy groups of the space $\L^{\tbb}(L)$. Combining it with Theorem \ref{thm:pi0-torus}, \ref{thm:pi0-Seifert}, \ref{thm:pi0-Sefert-Hopf} and \ref{thm:pi0-keychain}, we can conclude the proof. \qed

\section*{Declarations}
\textbf{Conflict of interest} The authors declare that they have no conflict of interest.

\bibliographystyle{plain}
\bibliography{references}
\end{document}